\documentclass[10pt,a4paper]{amsart}
\usepackage{amssymb,amsmath,amsopn,amsthm,graphicx}
\usepackage[all,2cell]{xy} \UseAllTwocells
\usepackage[parfill]{parskip}
\begin{document}

\makeatletter
\def\subsection{\@startsection{subsection}{3}%
  \z@{.5\linespacing\@plus.7\linespacing}{.1\linespacing}%
  {\rm\bf}}
\makeatother

\newtheorem{definition}{Definition}[subsection]
\newtheorem{definitions}[definition]{Definitions}
\newtheorem{deflem}[definition]{Definition and Lemma}
\newtheorem{lemma}[definition]{Lemma}
\newtheorem{pro}[definition]{Proposition}
\newtheorem{theorem}[definition]{Theorem}
\newtheorem{cor}[definition]{Corollary}
\newtheorem{cors}[definition]{Corollaries}
\theoremstyle{remark}
\newtheorem{remark}[definition]{Remark}
\theoremstyle{remark}
\newtheorem{remarks}[definition]{Remarks}
\theoremstyle{remark}
\newtheorem{notation}[definition]{Notation}
\theoremstyle{remark}
\newtheorem{example}[definition]{Example}
\theoremstyle{remark}
\newtheorem{examples}[definition]{Examples}
\theoremstyle{remark}
\newtheorem{dgram}[definition]{Diagram}
\theoremstyle{remark}
\newtheorem{fact}[definition]{Fact}
\theoremstyle{remark}
\newtheorem{illust}[definition]{Illustration}
\theoremstyle{remark}
\newtheorem{rmk}[definition]{Remark}
\theoremstyle{definition}
\newtheorem{que}[definition]{Question}
\theoremstyle{definition}
\newtheorem{conj}[definition]{Conjecture}

\newcommand{\stac}[2]{\genfrac{}{}{0pt}{}{#1}{#2}}
\newcommand{\stacc}[3]{\stac{\stac{\stac{}{#1}}{#2}}{\stac{}{#3}}}
\newcommand{\staccc}[4]{\stac{\stac{#1}{#2}}{\stac{#3}{#4}}}
\newcommand{\stacccc}[5]{\stac{\stacc{#1}{#2}{#3}}{\stac{#4}{#5}}}

\renewenvironment{proof}{\noindent {\bf{Proof.}}}{\hspace*{3mm}{$\Box$}{\vspace{9pt}}}
\title{Grothendieck Rings of Theories of Modules}


\maketitle
\begin{center}
		AMIT KUBER\footnote{Email \texttt{amit.kuber@postgrad.manchester.ac.uk}.
		Research partially supported by a School of Mathematics, University of Manchester Scholarship.},

		\medskip

		School of Mathematics, \ University of Manchester, \\
		Manchester M13 9PL, \ England.
	\end{center}

\begin{abstract}
The model-theoretic Grothendieck ring of a first order structure, as defined by Krajic\v{e}k and Scanlon, captures some combinatorial properties of the definable subsets of finite powers of the structure. In this paper we compute the Grothendieck ring, $K_0(M_\mathcal R)$, of a right $R$-module $M$, where $\mathcal R$ is any unital ring. As a corollary we prove a conjecture of Prest that $K_0(M)$ is non-trivial, whenever $M$ is non-zero. The main proof uses various techniques from the homology theory of simplicial complexes.
\end{abstract}

\section{Introduction}
In \cite{Kra}, Krajic\v{e}k and Scanlon introduced the concept of the model-theoretic Grothendieck ring of a structure. Amongst many other results, they proved that such a Grothendieck ring is nontrivial if and only if the definable subsets of the structure satisfy a version of the combinatorial pigeonhole principle, called the \textquotedblleft onto pigeonhole principle'' ($onto PHP$). Grothendieck rings have been studied for various rings and fields considered as models of a first order theory (see \cite{Kra}, \cite{Clu}, \cite{CluHask}, \cite{DenLoes1} and \cite{DenLoes2}) and they are found to be trivial in many cases (see \cite{Clu},\cite{CluHask}).

Prest conjectured that in stark contrast to the case of rings, for any ring $\mathcal R$, the Grothendieck ring of a nonzero right $\mathcal R$ module $M_\mathcal R$, denoted $K_0(M_\mathcal R)$, is nontrivial. Perera (\cite{Perera}) investigated the problem in his doctoral thesis and found that elementarily equivalent modules have isomorphic Grothendieck rings, which is not the case for general structures. He computed the Grothendieck ring for modules over semisimple rings and showed that they are polynomial rings in finitely many variables over the ring of integers.

In this paper we compute the Grothendieck ring for arbitrary modules and show that they are quotients of monoid rings $\mathbb Z[\mathcal X]$, where $\mathcal X$ is the multiplicative monoid of isomorphism classes of fundamental definable subsets of the module - the $pp$-definable subgroups. This is the content of the main theorem, theorem \ref{FINALgeneral}, which also describes the `invariants ideal' - the ideal of the monoid ring that codes indices of pairs of $pp$-definable subgroups. We further show (corollary \ref{MAINRESULTgeneral}) that there is a split embedding $\mathbb Z\rightarrow K_0(M)$, whenever the module $M$ is nonzero,  proving Prest's conjecture.

The proof of the main theorem uses inputs from various mathematical areas like model theory, algebra, combinatorics and algebraic topology. A special case of the main theorem (theorem \ref{FINAL}) is proved at the end of section \ref{spcasemult}. The special case assumes that the theory $T$ of the module $M$ satisfies the model theoretic condition $T=T^{\aleph_0}$. This condition is equivalent to the statement that the invariants ideal is trivial. The reader should note that the proof of the general case of the main theorem is not given in full detail since it develops along lines similar to the special case and uses only a few modifications to incorporate the invariants ideal.

The fundamental theorem of the model theory of modules (theorem \ref{PPET}) states that every definable set is a boolean combination of $pp$-definable sets, but such a boolean combination is far from being unique. We achieve a `uniqueness' result as a by-product of the theory we develop. We call this result the `cell decomposition theorem' (Theorem \ref{CDT1}) which states that definable sets can be represented uniquely using $pp$-definable sets provided the theory $T$ of the module satisfies $T=T^{\aleph_0}$. Though this theorem is not used directly in any other proof, its underlying idea is one of the most important ingredients of the main proof. Based on this idea, we define various classes of definable sets of increasing complexity, namely $pp$-sets, convex sets, blocks and cells. Our strategy to prove every result about a general definable set is to prove it first for convex sets, then blocks and then cells.

An important theme of the paper is the use of geometric and topological ideas in the setting of definable sets. We use the idea of a `neighbourhood' and `localization' to understand the structure of definable sets. We develop a notion of `connectedness' of a definable set in \ref{C} and prove theorem \ref{topconn} which clearly shows the analogy with its topological counterpart.

The main proof takes place at two different levels, which we name `local' and `global' following geometric intuition. We try to describe the ``shape'' of each definable set in terms of integer valued functions called `local characteristics'. These numbers are computed using Euler characteristics of various abstract simplicial complexes which code the ``local geometry'' of the given set. The local data is combined to get a family of integer valued functions, each of which is called a `global characteristic'. The global characteristics enjoy the property of being preserved under definable bijections. The family of such functions is indexed by the elements of the monoid $\mathcal X$ and the functions collate to give the necessary monoid ring.

The rest of the paper is organized as follows. Section \ref{prelim} contains the background material on Grothendieck rings and the model theory of modules. It also describes some important theorems in the homology theory of simplicial complexes. The core part of the proof of the special case is the content of section \ref{spcaseadd}. It introduces the terminology that we use and proves important facts about local and global characteristics. The highlights of this section are theorems \ref{t1} and \ref{t4}. Section \ref{spcasemult} contains proofs of the multiplicative properties of the global characteristics, completing the proof of the special case. Section \ref{gencase} introduces new terminology and the modifications in the proof of the special case necessary to handle the general case. Some applications of the main theorem are discussed in section \ref{appl}. The maps between modules which fit with model theory are called pure embeddings. We study their effect on Grothendieck rings in \ref{pure}. We also show the existence of Grothendieck rings containing nontrivial torsion elements in \ref{tors}. The cell decomposition theorem is proved in \ref{CDT}, whereas the discussion on connectedness is included in \ref{C}. We conclude the paper with section \ref{rmkcom} which contains further remarks on the technique of the proof and mentions some directions for further research in this area.

\section{Preliminaries}\label{prelim}
\subsection{Semirings and Grothendieck rings}\label{SGR}

We recall the notion of a semiring and  how to construct a ring in a canonical fashion from a given semiring. A detailed exposition on this material can be found in \cite{Lee}.

Let $L_{ring}=\langle0,1,+,\cdotp\rangle$ be the language of rings.
\begin{definitions}
Any $L_{ring}$ structure $S$ satisfying the following conditions is a commutative \textbf{semiring} with unity.
\begin{itemize}
	\item $(S,+,0)$ is a commutative monoid
	\item $(S,\cdotp,1)$ is a commutative monoid
	\item $a\cdotp 0=0$ for all $a\in S$
	\item multiplication ($\cdotp$) distributes over addition ($+$)
\end{itemize}

A \textbf{semiring homomorphism} is an $L_{ring}$-homomorphism.

A semiring $S$ is said to be \textbf{cancellative} if $a+c=b+c\ \Rightarrow\ a=b$ for all $a,b,c\in S$.
\end{definitions}

In \cite{Lee}, a cancellative semiring is called a \emph{halfring}. All the semirings considered here are commutative semirings with unity, allowing the possibility $0=1$.

\begin{definition}
A binary relation $\thicksim$ on a semiring $S$ is said to be a \textbf{congruence relation} if the following properties hold.
\begin{itemize}
	\item $\thicksim$ is an equivalence relation
	\item $a\thicksim b,c\thicksim d$ for $a,b,c,d\in S$ $\Rightarrow (a+c)\thicksim(b+d),a\cdotp c\thicksim b\cdotp d$
\end{itemize}
\end{definition}

There is a canonical way of constructing a cancellative semiring from any semiring $S$ as stated in the following theorem.
\begin{theorem}\label{QUOCONST}\textbf{Quotient construction}:
Let $S$ be a semiring and let $\thicksim$ be the binary relation defined as follows.
\begin{equation}\label{CANCEL}
For\ a,b\in S,\ a\thicksim b\ \Leftrightarrow\ \exists c\in S,\ a+c=b+c
\end{equation}
Then $\thicksim$ is a congruence relation. If $\tilde{a}$ denotes the $\thicksim$ equivalence class of $a\in S$, then $\tilde{S}:=\{\tilde{a}:a\in S\}$ is a cancellative semiring with respect to the induced addition and multiplication operations. There is a surjective semiring homomorphism $q:S\rightarrow\tilde{S}$ given by $a\mapsto \tilde{a}$. Furthermore, given any cancellative semiring $T$ and a semiring homomorphism $f:S\rightarrow T$, there exists a unique semiring homomorphism $\tilde{f}:\tilde{S}\rightarrow T$ such that the diagram $\xymatrix{{S}\ar[rr] ^{q}\ar[dr]^f & & {\tilde{S}}\ar@{->}[dl]^{\exists !\tilde{f}} \\ & {T}}$ commutes.
\end{theorem}

One can embed a cancellative semiring in a ring in a canonical fashion as stated in the following theorem.
\begin{theorem}\label{GRCONSTR}
\textbf{Ring of Differences for a Cancellative Semiring}:
Let $R$ denote a cancellative semiring and let $E$ denote the binary relation on the set $R\times R$ of ordered pairs of elements from $R$ defined as follows.
\begin{equation}\label{NEGADD}
For\ (a,b),(c,d)\in R\times R,\ (a,b) E (c,d)\ \Leftrightarrow\ a+d=b+c
\end{equation}
Then $R$ is an equivalence relation. If $(a,b)_E$ denotes the $E$-equivalence class of $(a,b)$, then the quotient structure $(R\times R)/E:=\{(a,b)_E:(a,b)\in R\times R\}$ is a ring with respect to the operations given by
\begin{eqnarray}\label{RINGOPER}
(a,b)_E+(c,d)_E&:=&(a+c,b+d)_E\\
(a,b)_E\cdotp(c,d)_E&:=&(a\cdotp c+b\cdotp d,a\cdotp d+b\cdotp c)_E\\
-(a,b)_E&:=&(b,a)_E
\end{eqnarray}
for $(a,b)_E,(c,d)_E\in(R\times R)/E$. We denote the ring $(R\times R)/E$ by $K_0(R)$ following the conventions of K-theory. The semiring $R$ can be embedded into the ring $K_0(R)$ by the semiring homomorphism $i$ given by $a\mapsto (a,0)$. Furthermore, given any ring $T$ and a semiring homomorphism $g:R\rightarrow T$, there exists a unique ring homomorphism $\overline g:K_0(R)\rightarrow T$ such that the diagram $\xymatrix{{R}\ar[rr] ^{i}\ar[dr]^g & & {K_0(R)}\ar@{->} [dl]^{\exists !\overline{g}} \\ & {T}}$ commutes.
\end{theorem}

Note that each of the $E$-equivalence classes of the elements from $R\times R$, as constructed in the previous theorem, contains a pair of the form $(a,0)$ or $(0,a)$ for some $a\in R$.

For a semiring $S$, let $K_0(S)$ denote the ring $K_0(\tilde S)$ for simplicity, where $\tilde S$ is the cancellative semiring obtained from $S$ as stated in the theorem \ref{QUOCONST} and let the canonical map $S\rightarrow K_0(S)$ be denoted by $\eta_S$. We finally note the following result which combines the previous two theorems.
\begin{cor}\label{GrRngAdj}
A semiring $S$ can be embedded in a ring if and only if $S$ is cancellative. Given any ring $T$ and a semiring homomorphism $g:S\rightarrow T$, there exists a unique ring homomorphism $\overline g:K_0(S)\rightarrow T$ such that the diagram $\xymatrix{S\ar[rr] ^{\eta_S}\ar[dr]^g & & {K_0(S)}\ar@{->}[dl]^{\exists !\overline{g}} \\ & {T}}$ commutes.
\end{cor}

This result can be stated in category theoretic language as follows. Let $\mathbf{CSemiRing}$ denote the category of commutative semirings with unity and semiring homomorphisms preserving unity. Let $\mathbf{CRing}$ denote its full subcategory consisting of commutative rings with unity and let $I:\mathbf{CRing}\rightarrow\mathbf{CSemiRing}$ be the inclusion functor. Then $I$ admits a left adjoint, namely $K_0:\mathbf{CSemiRing}\rightarrow\mathbf{CRing}$. For each semiring $S$, the ring $K_0(S)$ is called the Grothendieck Ring constructed from $S$. If $\eta$ is the unit of the adjunction, the diagram in the previous corollary represents the universal property of the adjunction.

\subsection{Grothendieck rings of first order structures}\label{GRFOS}

We aim to introduce the notion of the model theoretic Grothendieck ring of a first order structure in this section. This account is based on \cite{Kra}. After setting some background in model theory, we state how to construct the semiring of definable isomorphism classes of definable subsets of finite cartesian powers of the given structure $M$. Following the method described in the previous section, we then construct the Grothendieck ring $K_0(M)$.

Let $L$ denote any language and $M$ denote any first order $L$-structure. The term definable will always mean definable with parameters from $M$.

\begin{definitions}
For every $n\geq 1$, we define $\mathrm{Def}(M^n)$ to be the collection of all definable subsets of $M^n$. We define $\overline{\mathrm{Def}}(M):=\bigcup_{n\geq 1}\mathrm{Def}(M^n)$.
\end{definitions}

\begin{definition}\label{defiso}
We say that two definable sets $A,B\in\overline{\mathrm{Def}}(M)$ are \textbf{definably isomorphic} if there exists a definable bijection between them, i.e., a bijection $f:A\rightarrow B$ such that the graph $Gr(f)\in\overline{\mathrm{Def}}(M)$. This is an equivalence relation on $\overline{\mathrm{Def}}(M)$ and the equivalence class of a set $A$ is denoted by $[A]$. We use $\widetilde{\mathrm{Def}}(M)$ to denote the set of all equivalence classes with respect to this relation. We use $[-]:\overline{\mathrm{Def}}(M)\rightarrow\widetilde{\mathrm{Def}}(M)$ to denote the surjective map defined by $A\mapsto[A]$.
\end{definition}

We can regard $\widetilde{\mathrm{Def}}(M)$ as an $L_{ring}$-structure. In fact, it is a semiring with respect to the operations defined as follows.
\begin{itemize}
	\item $0 := [\emptyset]$
	\item $1 := [\{*\}]$ for any singleton subset $\{*\}$ of $M$
	\item $[A]+[B] := [A'\sqcup B']$ for $A'\in[A],B'\in[B]$ such that $A'\cap B'=\emptyset$
	\item $[A]\cdotp[B] := [A\times B]$
\end{itemize}
(NB: We use $\sqcup$ to denote disjoint unions.)

Now we are ready to give the most important definition.
\begin{definition}
We define the \textbf{model-theoretic Grothendieck ring of the first order structure} $M$, denoted by $K_0(M)$, to be the ring $K_0(\widetilde{\mathrm{Def}}(M))$ obtained from corollary \ref{GrRngAdj}, where the semiring structure on $\widetilde{\mathrm{Def}}(M)$ is as defined above.
\end{definition}

This ring captures the definable combinatorics of the structure $M$. We are interested to know whether $K_0(M)=\{0\}$. It is useful to consider some definable combinatorial aspects to tackle this problem.
\begin{definition}
We say that an infinite structure $M$ satisfies the \textbf{pigeonhole principle} if for each $A\in\overline{\mathrm{Def}}(M)$, each definable injection $f:A\rightarrowtail A$ is an isomorphism. We write this as $M\vDash PHP$.
\end{definition}

This condition is very strong to be true for many structures. As an example, consider the additive group of integers $\mathbb Z$ in the language of abelian groups. The function $\mathbb Z\xrightarrow{(-)\times 2}\mathbb Z$ is a definable injection but not an isomorphism. So it is useful to consider some weaker forms. Though there are several of them (see \cite{Kra}), we note the one important for us.
\begin{definition}
We say that an infinite structure $M$ satisfies the \textbf{onto pigeonhole principle} if for each $A\in\overline{\mathrm{Def}}(M)$ and each definable injection $f:A\rightarrowtail A$, we have $f(A)\neq A\setminus\{a\}$ for any $a\in A$. We write this as $M\vDash ontoPHP$.
\end{definition}

The following proposition gives the necessary and sufficient condition for $K_0(M)$ to be nontrivial (i.e. $0\neq 1$ in $K_0(M)$). We include a proof for the sake of completeness.
\begin{pro}
Given any infinite structure $M$, $K_0(M)\neq\{0\}$ if and only if $M\vDash ontoPHP$.
\end{pro}

\begin{proof}
Recall the construction of the cancellative semiring from (\ref{CANCEL}). The condition $0=1$ in $K_0(M)$ is thus equivalent to the statement that for some $A\in\overline{\mathrm{Def}}(M)$, we have $0+[A]=1+[A]$. This is precisely the statement that $M\nvDash ontoPHP$.
\end{proof}

\textbf{A brief survey of known Grothendieck Rings:} Very few examples of Grothendieck rings are known in general. If $M$ is a finite structure, then $K_0(M)\cong\mathbb Z$. Kraji\v{c}ek and Scanlon have shown in \cite[Example\,3.6]{Kra} that $K_0(\mathbb R)\cong\mathbb Z$ using the dimension theory and cell decomposition theorem for o-minimal structures, where $\mathbb R$ denotes the real closed field. Cluckers and Haskell (\cite{Clu}, \cite{CluHask}) proved that the fields of p-adic numbers have trivial Grothendieck rings, by constructing definable bijections from a set to the same set minus a point. Denef and Loeser (\cite{DenLoes1},\cite{DenLoes2}) have found that the Grothendieck ring $K_0(\mathbb C)$ of the field $\mathbb C$ of complex numbers regarded as an $L_{ring}$-structure admits the ring $\mathbb Z[X;Y]$ as a quotient. Kraji\v{c}ek and Scanlon have strengthened this result and shown that $K_0(\mathbb C)$ contains an algebraically independent set of size continuum, and hence the ring $\mathbb Z[X_i:i\in\mathfrak{c}]$ embeds into $K_0(\mathbb C)$. Perera showed in \cite[Theorem\,4.3.1]{Perera} that $K_0(M)\cong\mathbb Z[X]$ whenever $M$ is an infinite module over an infinite division ring. Prest conjectured \cite[Ch.\,8,\,Conjecture A]{Perera} that $K_0(M)$ is nontrivial for all nonzero right $\mathcal R$-modules $M$. We prove that $K_0(M)$ is actually a quotient of a monoid ring and furthermore it is nontrivial. Most of the paper is devoted to the proof of this statement.

\subsection{Euler characteristic of simplicial complexes}\label{ECSC}

We introduce the concept of an abstract simplicial complex and a couple of ways to calculate its Euler characteristic. We also state some important results in the homology theory of simplicial complexes. The material on homology and relative homology presented in this section is taken from \cite[II.4]{FerPic}. This theory provides the basis for the analysis of `local characteristics' in \ref{LC}.

\begin{definition}
An \textbf{abstract simplicial complex} is a pair $(X,\mathcal{K})$ where $X$ is a finite set and $\mathcal{K}$ is a collection of subsets of $X$ satisfying the following properties.
\begin{itemize}
	\item $\emptyset\notin\mathcal{K}$
	\item $\{x\}\in\mathcal{K}$ for each $x\in X$
	\item if $F\in\mathcal{K}$ and $\emptyset\neq F'\subsetneq F$, then $F'\in\mathcal{K}$
\end{itemize}
\end{definition}
We usually identify the simplicial complex $(X,\mathcal{K})$ with $\mathcal{K}$. The elements $F\in\mathcal{K}$ are called the \textbf{faces} of the complex and the singleton faces are called the \textbf{vertices} of the complex. We use $\mathcal V(\mathcal K)$ to denote the set of vertices of $\mathcal K$.

Let $\Delta^k:=\mathbb{P}([k+1])\setminus\{\emptyset\}$ denote the \textbf{standard $k$-simplex}, where $\mathbb P$ denotes the power set operator and $[k+1]=\{1,2,\hdots,k+1\}$ for $k\geq 0$. We define the \textbf{geometric realization} of the standard $k$-simplex, denoted $|\Delta^k|$, to be the set of all points of $\mathbb R^{k+1}$ which can be expressed as a convex linear combination of the standard basis vectors of $\mathbb R^{k+1}$. In fact we can associate to every abstract simplicial complex a topological space $|\mathcal K|$, called its geometric realization. This topological space is constructed by `gluing together' the geometric realizations of its simplices.

We assign dimension to every face $F\in\mathcal{K}$ by stating $\mathrm{dim} F:=|F|-1$ and we say that the \textbf{dimension of the complex} is the maximum of the dimensions of its faces.

\begin{definition}\label{Euler}
We define the \textbf{Euler characteristic} of the complex $\mathcal{K}$, denoted $\chi(\mathcal{K})$, to be the integer $\Sigma_{n=0}^{\mathrm{dim}\mathcal{K}}(-1)^n v_n$ where $v_n$ is the number of faces in $\mathcal{K}$ with dimension $n$.
\end{definition}

It is easy to check that $\chi(\Delta^k)=1$ for each $k\geq 0$. Since we also allow our complex to be empty, we define $\chi(\emptyset):=0$ though $\mathrm{dim}\emptyset$ is undefined.

There is yet another way to obtain the Euler characteristics of simplicial complexes, via homology. The word homology will always mean simplicial homology with integer coefficients in this paper. If $b_n$ denotes the $n^{th}$ Betti number of the simplicial complex $\mathcal{K}$ (i.e. the rank of the $n^{th}$ homology group $H_n(\mathcal{K})$), then we have the identity $\chi(\mathcal{K})=\Sigma_{n=0}^\infty(-1)^n b_n$ where the sum on the right hand side is finite. We use the notation $C_*(\mathcal K)$ to denote the chain complex $C_n(\mathcal K)_{n\geq 0}$ and $H_*(\mathcal K)$ to denote the chain complex $(H_n(\mathcal K))_{n\geq0}$, where $C_n(\mathcal K)$ is the free abelian group generated by the set of $n$-simplices in $\mathcal K$.

The following result states that homology is a homotopy invariant. It will be useful in proving a key result (proposition \ref{p1}).
\begin{theorem}\label{HTPYINV}
If $\mathcal K_1$ and $\mathcal K_2$ (meaning, their geometric realizations) are homotopy equivalent, then $H_*(\mathcal K_1)\cong H_*(\mathcal K_2)$.
\end{theorem}

The definition of Euler characteristic in terms of Betti numbers gives the following corollary.
\begin{cor}\label{EULHTPY}
If $\mathcal K_1$ and $\mathcal K_2$ are homotopy equivalent, then $\chi(\mathcal K_1)=\chi(\mathcal K_2)$.
\end{cor}

The homology groups $H_n(\mathcal K)$, for $n\geq 1$, calculate the number of ``$n$-dimensional holes'' in the geometric realization of the complex $\mathcal K$. But sometimes it is important to ignore the data present in a smaller part of the given structure. This can be done in two ways, viz. using the cone construction for a subcomplex or by using relative homology. Given a complex $\mathcal{K}$ and a subcomplex $\mathcal{Q}\subseteq\mathcal{K}$, we write $\mathcal K\cup\mathrm{Cone}(\mathcal{Q})$ for the simplicial complex whose vertex set is $\mathcal V(\mathcal K)\cup\{x\}$, where $x\notin\mathcal V(\mathcal K)$, and the faces are $\mathcal{K}\cup\{\{x\}\cup F:F\in\mathcal{Q}\}$. We say that $x$ is the \textbf{apex} of the cone. In the same situation, we use the notation $H_n(\mathcal K;\mathcal{Q})$ to denote the $n^{th}$ homology of $\mathcal K$ relative to $\mathcal{Q}$.

The following theorem connects the relative homologies with the homologies of the original complexes.
\begin{theorem}(see \cite[Theorem\,2.16]{Hatcher})\label{LONGEXACT}
Given a pair of simplicial complexes $\mathcal{Q}\subset\mathcal K$, we have the following long exact sequence of homologies.\\
\begin{equation*}
\cdots\rightarrow H_n(\mathcal{Q})\rightarrow H_n(\mathcal{K})\rightarrow H_n(\mathcal K;\mathcal{Q})\rightarrow H_{n-1}(\mathcal{Q})\rightarrow\cdots\rightarrow H_0(\mathcal K;\mathcal{Q})\rightarrow 0	
\end{equation*}
\end{theorem}

We shall also make use of the following result.
\begin{theorem}
Given a pair of simplicial complexes $\mathcal{Q}\subseteq\mathcal{K}$, we have $H_n(\mathcal K;\mathcal{Q}) \cong H_n(\mathcal K\cup\mathrm{Cone}(\mathcal{Q}))$ for $n\geq 1$ and $H_0(\mathcal K\cup \mathrm{Cone}(\mathcal{Q}))\cong H_0(\mathcal K;\mathcal{Q})\oplus \mathbb Z$.
\end{theorem}

\begin{illust}
Let $\mathcal K=\{\{1\},\{2\},\{3\},\{1,2\},\{2,3\}\}$ and $\mathcal Q$ denote the subcomplex $\{\{1\},\{3\}\}$. Then
\begin{eqnarray*}
H_n(\mathcal K)&=&\begin{cases}
				                \mathbb Z, &\mbox{if } n=0,\\
								0, &\mbox{otherwise}
								\end{cases}\\
H_n(\mathcal Q)&=&\begin{cases}
								\mathbb Z\oplus\mathbb Z, &\mbox{if } n=0,\\
								0, &\mbox{otherwise}
								\end{cases}\\
H_n(\mathcal K;\mathcal Q)&=&\begin{cases}
								\mathbb Z, &\mbox{if } n=1,\\
								0, &\mbox{otherwise}
								\end{cases}\\
H_n(\mathcal K\cup\mathrm{Cone}(\mathcal Q))&=&\begin{cases}
								\mathbb Z, &\mbox{if } n=0,1,\\
								0, &\mbox{otherwise}
								\end{cases}
\end{eqnarray*}
\end{illust}

Combining the previous two results with the definition of Euler characteristic, we get
\begin{cor}\label{HMLGCONE}
For a pair of simplicial complexes $\mathcal{Q}\subseteq\mathcal K$, $\chi(\mathcal K\cup \mathrm{Cone}(\mathcal{Q}))+\chi(\mathcal{Q})=\chi(\mathcal K)+1$.
\end{cor}

\subsection{Products of simplicial complexes}\label{PRODSIMPCOMP}

We define various products of simplicial complexes and study their interrelations. The inclusion-exclusion principle stated in lemma \ref{ecdpl} is equivalent to the statement that `local characteristics are multiplicative' (lemma \ref{localcharmult}).

Let $\mathcal K$ and $\mathcal Q$ be two simplicial complexes with vertex sets $\mathcal{V}(\mathcal K)$ and $\mathcal{V}(\mathcal Q)$ respectively and let $\pi_1:\mathcal{V}(\mathcal K)\times\mathcal{V}(\mathcal Q)\rightarrow \mathcal{V}(\mathcal K)$ and $\pi_2:\mathcal{V}(\mathcal K)\times\mathcal{V}(\mathcal Q)\rightarrow \mathcal{V} (\mathcal Q)$ denote the projection maps. We define two simplicial complexes with the vertex set $\mathcal V(\mathcal K)\times\mathcal V(\mathcal Q)$. The following product is defined in \cite[\S 3]{EilZil}.

\begin{definition}
The \textbf{simplicial product} $\mathcal K\vartriangle\mathcal Q$ of two simplicial complexes $\mathcal K$ and $\mathcal Q$ is a simplicial complex with vertex set $\mathcal V(\mathcal K)\times\mathcal V(\mathcal Q)$ where a nonempty set $F\subseteq\mathcal V(\mathcal K)\times\mathcal V(\mathcal Q)$ is a face of $\mathcal K\vartriangle\mathcal Q$ if and only if $\pi_1(F)\in\mathcal K$ and $\pi_2(F)\in\mathcal Q$.
\end{definition}

\begin{definition}
The \textbf{disjunctive product} $\mathcal K\boxtimes\mathcal Q$ of two simplicial complexes $\mathcal K$ and $\mathcal Q$ is a simplicial complex with vertex set $\mathcal V(\mathcal K)\times\mathcal V(\mathcal Q)$ where a nonempty set $F\subseteq\mathcal V(\mathcal K)\times\mathcal V(\mathcal Q)$ is a face of $\mathcal K\boxtimes\mathcal Q$ if and only if $\pi_1(F)\in\mathcal K$ or $\pi_2(F)\in\mathcal Q$.
\end{definition}

Observe that the previous two definitions are identical except for the word `and' in the former is replaced by the word `or' in the latter. Thus the simplicial product $\mathcal K\vartriangle\mathcal Q$ is always contained in the disjunctive product $\mathcal K\boxtimes\mathcal Q$.

\begin{illust}\label{spdp}
Let $\mathcal K=\{\{1\},\{2\}\}$ denote the complex consisting precisely of two vertices. Then $\mathcal K\vartriangle\mathcal K$ contains only the vertices of the `square' $\mathcal K\boxtimes\mathcal K$ given by $\{\{(1,1)\},\{(1,2)\},\{(2,1)\},\{(2,2)\},\{(1,1),(1,2)\},\{(2,1),(2,2)\},\{(1,1),(2,1)\},\\\{(1,2),(2,2)\}\}$. For each $k\geq 0$ the complex $\mathcal K\vartriangle\Delta^k$ is the union of two disjoint copies of $\Delta^k$, whereas the complex $\mathcal K\boxtimes\Delta^k$ is a copy of $\Delta^{2k+1}$.
\end{illust}

The main aim of this section is to prove the following lemma about the Euler characteristic of the disjunctive product.

\begin{lemma}\label{ecdpl}
The Euler characteristics of two simplicial complexes $\mathcal K$ and $\mathcal Q$ satisfy
\begin{equation}\label{ecdp}
\chi(\mathcal K\boxtimes\mathcal Q)=\chi(\mathcal K)+\chi(\mathcal Q)-\chi(\mathcal K)\chi(\mathcal Q).
\end{equation}
\end{lemma}

\begin{illust}
Let $\mathcal K$ be as defined in \ref{spdp}. Then we observe that $\chi(\mathcal K)=2$. Since $\mathcal K\boxtimes\mathcal K$ contains $4$ vertices and $4$ edges, we get $\chi(\mathcal K\boxtimes\mathcal K)=0=2\chi(\mathcal K)-\chi(\mathcal K)^2$ verifying equation (\ref{ecdp}) in this case.
\end{illust}

The proof of the lemma uses tensor products of chain complexes.
\begin{definition}
Let $C_*=\{C_n,\partial_n\}_{n\geq 0}$ and $D_*=\{D_n,\delta_n\}_{n\geq 0}$ denote two bounded chain complexes of abelian groups. The \textbf{tensor product complex} $C_*\otimes D_*=\{(C_*\otimes D_*)_n,d_n\}_{n\geq 0}$ is defined by
\begin{eqnarray*}
(C_*\otimes D_*)_n&=&\bigoplus_{i+j=n}C_i\otimes D_j,\\
d_n(a_i\otimes b_j)&=&\partial_i(a_i)\otimes b_j+(-1)^{i}a_i\otimes\delta_j(b_j).
\end{eqnarray*}
\end{definition}

\begin{illust}\label{tensor}
We compute the tensor product $C_*(\partial\Delta^2)\otimes C_n(\Delta^1)$ as an example, where $\partial\Delta^2$ denotes the boundary of $\Delta^2$.
\begin{eqnarray*}
C_n(\partial\Delta^2)&=&\begin{cases}
								\mathbb Z\oplus\mathbb Z\oplus\mathbb Z, &\mbox{if } n=0,1,\\
								0, &\mbox{otherwise}
								\end{cases}\\
C_n(\Delta^1)&=&\begin{cases}
								\mathbb Z\oplus\mathbb Z, &\mbox{if } n=0,\\
								\mathbb Z, &\mbox{if } n=1,\\
                                0, &\mbox{otherwise}
								\end{cases}\\
C_n(\partial\Delta^2)\otimes C_n(\Delta^1)&=&\begin{cases}
                                \oplus_{i=1}^6\mathbb Z, &\mbox{if } n=0,\\
                                \oplus_{i=1}^9\mathbb Z, &\mbox{if } n=1,\\
								\oplus_{i=1}^3\mathbb Z, &\mbox{if } n=2,\\
								0, &\mbox{otherwise}
								\end{cases}
\end{eqnarray*}
\end{illust}

There is yet one more product of simplicial complexes, viz., the cartesian product, defined in the literature (see \cite{EilZil}). We avoid its use by dealing with the product of geometric realizations (with the product topology). The homology of such (finite) product spaces is easily computed using triangulation. We first note that the Euler characteristic is multiplicative.
\begin{pro}\label{euprod}(see \cite[p.205,\,Ex.\,B.4]{Spa})
Let $\mathcal K$ and $\mathcal Q$ be any simplicial complexes. Then
\begin{equation*}\chi(|\mathcal K|\times|\mathcal Q|)=\chi(\mathcal K)\chi(\mathcal Q).\end{equation*}
\end{pro}

A famous theorem of Eilenberg and Zilber (see \cite{EilZil}) connects the homologies of two semi-simplicial complexes (a term used in 1950 that includes the class of simplicial complexes) with that of their cartesian product. We state this result below using the cartesian product of their geometric realizations. More details can be found in \cite[\S 2.1]{Hatcher} and \cite[\S III.6]{FerPic}.

\begin{theorem}\label{EZT}(see. \cite[\S III.6.2]{FerPic})
Let $\mathcal K$ and $\mathcal Q$ be any two simplicial complexes. Then we have \\
$H_*(|\mathcal K|\times|\mathcal Q|)\cong H_*(C_*(\mathcal K)\otimes C_*(\mathcal Q))$.
\end{theorem}

Furthermore, Eilenberg and Zilber state the following corollary of the previous theorem in \cite{EilZil}.
\begin{cor}\label{simpprod}
Let $\mathcal K$ and $\mathcal Q$ be any two simplicial complexes. Then\\
$H_*(\mathcal K\vartriangle\mathcal Q)\cong H_*(C_*(\mathcal K)\otimes C_*(\mathcal Q))$.
\end{cor}

\begin{illust}
We continue the example in \ref{tensor}. The computation of the boundary operators yields
\begin{equation*}
H_n(C_*(\partial\Delta^2)\otimes C_*(\Delta^1))=\begin{cases}
                \mathbb Z, &\mbox{if } n=0,1,\\
								0, &\mbox{otherwise}
								\end{cases}
\end{equation*}
The space $|\partial\Delta^2|\times|\Delta^1|$ is a cylinder which is homotopy equivalent to $S^1$. Hence $H_n(|\partial\Delta^2|\times|\Delta^1|)=\mathbb Z$ for $n=0,1$ and is zero for other values of $n$. This completes the illustration of theorem \ref{EZT}.

Furthermore the complex $\partial\Delta^2\vartriangle\Delta^1$ is the union of three copies of $\Delta^3$ each of which shares exactly one edge (i.e. a copy of $\Delta^1$) with every other copy and these three edges are pairwise disjoint. It can be easily see that this complex (i.e. its geometric realization) is homotopy equivalent to the circle and hence the conclusions of the corollary \ref{simpprod} hold.
\end{illust}

\begin{proof} (Lemma \ref{ecdpl})
We first observe that there is an embedding of simplicial complexes $\iota_1:\mathcal K\vartriangle(\Delta^{|\mathcal V(\mathcal Q)|-1})\rightarrow\mathcal K\boxtimes\mathcal Q$ induced by some fixed enumeration of $\mathcal V(\mathcal Q)$. Similarly there is an embedding $\iota_2:(\Delta^{|\mathcal V(\mathcal K)|-1})\vartriangle\mathcal Q\rightarrow\mathcal K\boxtimes\mathcal Q$ induced by some fixed enumeration of $\mathcal V(\mathcal K)$. Furthermore, the intersection $\iota_1(\mathcal K\vartriangle(\Delta^{|\mathcal V(\mathcal Q)|-1}))\cap\iota_2((\Delta^{|\mathcal V(\mathcal K)|-1})\vartriangle\mathcal Q)$ is precisely the complex $\mathcal K\vartriangle\mathcal Q$.

This gives us, using the counting definition of the Euler characteristics, that the identity
\begin{equation}\label{ecintermediate}
\chi(\mathcal K\boxtimes\mathcal Q)=\chi(\mathcal K\vartriangle(\Delta^{|\mathcal V(\mathcal Q)|-1}))+\chi((\Delta^{|\mathcal V(\mathcal K)|-1})\vartriangle\mathcal Q)-\chi(\mathcal K\vartriangle\mathcal Q)
\end{equation}
holds.

If we can prove that $\chi(\mathcal K\vartriangle\mathcal Q)=\chi(\mathcal K)\chi(\mathcal Q)$ for any simplicial complexes $\mathcal K$ and $\mathcal Q$, then (\ref{ecintermediate}) will yield (\ref{ecdp}) since $\chi(\Delta^k)=1$ for each $k\geq 0$.

Now we have
$H_*(\mathcal K\vartriangle\mathcal Q)\cong H_*(C_*(\mathcal K)\otimes C_*(\mathcal Q))\cong H_*(C_*(|\mathcal K|\times|\mathcal Q|))$, where the first isomorphism is by \ref{simpprod} and the second by \ref{EZT}.

Hence we have $\chi(\mathcal K\vartriangle\mathcal Q)=\chi(|\mathcal K|\times|\mathcal Q|)=\chi(\mathcal K)\chi(\mathcal Q)$ by \ref{euprod} as required. This completes the proof.
\end{proof}

\subsection{Model theory of modules}\label{MTM}

We introduce the terminology and some basic results in the model theory of modules in this section. A detailed exposition can be found in \cite{PreBk}. Instead of working with formulas all the time, we fix a structure and work with the definable subsets of finite cartesian powers of that structure.

Let $\mathcal{R}$ be a fixed ring with unity. Then every right $\mathcal{R}$-module $M$ is a structure for the first order language $L_\mathcal{R}=\langle 0,+,-,m_r:r\in\mathcal{R}\rangle$, where each $m_r$ is a unary function symbol representing the action of right multiplication by the element $r$. When we are working in a fixed module $M$, we usually write the element $m_r(a)$ in formulas as $ar$ for each $a\in M$.

First we note the following result of Perera which states that the Grothendieck ring of a module is an invariant of its theory. A proof of this proposition can be found at the end of section \ref{gencase} as a corollary of theorem \ref{FINALgeneral}.

\begin{pro}(see \cite[Corollary\,5.3.2]{Perera})\label{eleequivmod}
Let $M$ and $N$ be two right $\mathcal R$-modules such that $M\equiv N$, then $K_0(M)\cong K_0(N)$.
\end{pro}

Let us fix a right $\mathcal{R}$-module $M$. Then every definable subset of $M^n$ for any $n\geq 1$ can be expressed in terms of certain fundamental definable subsets of $M^n$. In order to state this partial quantifier elimination result, we first define the formulas which define these fundamental subsets.

\begin{definition}
A \textbf{positive primitive formula} (\textbf{pp-formula} for short) is a formula in the language $L_\mathcal{R}$ which is equivalent to one of the form
\begin{equation*}
\phi(x_1,x_2,\hdots,x_n)=\exists y_1\exists y_2\hdots\exists y_m\bigwedge_{i=1}^t\left(\sum_{j=1}^n x_j r_{ij}+\sum_{k=1}^m y_ks_{ik}+c_i=0\right),
\end{equation*}
where $r_{ij},s_{ik}\in\mathcal R$ and the $c_i$ are parameters from $M$.
\end{definition}

A subset of $M^n$ which is defined by a $pp$-formula (with parameters) will be called a $pp$-set. If a subgroup of $M^n$ is $pp$-definable, then its cosets are also $pp$-definable. The following lemma is well known and a proof can be found in \cite[Corollary\,2.2]{PreBk}.

\begin{lemma}
Every parameter-free $pp$-formula $\phi(\overline x)$ defines a subgroup of $M^n$, where $n$ is the length of $\overline x$. If $\phi(\overline x)$ contains parameters from $M$, then it defines either the empty set or a coset of a $pp$-definable subgroup of $M^n$. Furthermore, the conjunction of two $pp$-formulas is (equivalent to) a $pp$-formula.
\end{lemma}

Let $\mathcal{L}_n$ denote the meet-semilattice of all $pp$-subsets of $M^n$ ordered by the inclusion relation $\subseteq$. We will use the notation $\mathcal{L}_n(M_\mathcal{R})$, specifying the module, when we work with more than one module at a time.

\begin{definition}
Let $M$ be a right $\mathcal R$-module and let $A,B\in\mathcal L_n$ be subgroups. We define the invariant $\mathrm{Inv}(M,A,B)$ to be the index $[A:A\cap B]$ if this if finite or $\infty$ otherwise.
\end{definition}

An \textbf{invariants condition} is a statement that a given invariant is greater than or equal to or less than a certain number. These invariant conditions can be expressed as sentences in $L_\mathcal R$. An \textbf{invariants statement} is a finite boolean combination of invariants conditions.

We are now ready to state the promised fundamental theorem of the model theory of modules.
\begin{theorem}(see \cite{Baur})\label{PPET}
Let $T$ be the theory of right $\mathcal R$-modules and $\phi(\overline x)$ be an $L_\mathcal R$ formula (possibly with parameters). Then we have
\begin{equation*}
T\vDash \forall \overline x (\phi(\overline x)\leftrightarrow\bigvee_{i=1}^m\left(\psi_i(\overline x)\wedge\bigwedge_{j=1}^{l_i}\neg\chi_{ij}(\overline x)\right)\wedge I),
\end{equation*}
where $I$ is an invariants statement and $\psi_i(\overline x),\chi_{ij}(\overline x)$ are $pp$-formulas.
\end{theorem}

We may assume that $\chi_{ij}(M)\subseteq\psi_i(M)$ for each value of $i$ and $j$, otherwise we redefine $\chi_{ij}$ as $\chi_{ij}\wedge\psi_i$. When we work in a complete theory, the invariants statements will vanish and hence we get the following form.
\begin{theorem}
For each $n\geq 1$, every definable subset of $M^n$ can be expressed as a finite boolean combination of $pp$-subsets of $M^n$.
\end{theorem}

Using this result together with the meet-semilattice structure of $\mathcal{L}_n$, we can express each definable subset of $M^n$ in a ``disjunctive normal form'' of $pp$-sets. Expressing a definable set as a disjoint union helps to break it down to certain low complexity fragments, each of which has a specific shape given by the normal form. A proof of this result can be found in \cite[Lemma\,3.2.1]{Perera}.
\begin{lemma}\label{REP}
Every definable subset of $M^n$ can be written as $\bigsqcup_{i=1}^t (A_i\setminus(\bigcup_{j=1}^{s_i} B_{ij}))$ for some $A_i,B_{ij}\in\mathcal{L}_n$.
\end{lemma}

The following lemma is one of the important tools in our analysis.
\begin{lemma}\label{NL}
\textbf{Neumann's Lemma} (see \cite[Theorem\,2.12]{PreBk})\\
If $H$ and $G_i$ are subgroups of some group $(K,+)$ and a coset of $H$ is covered by a finite union of cosets of the $G_i$, then this coset of $H$ is in fact covered by the union of just those cosets of $G_i$ where $G_i$ is of finite index in $H$, i.e. where $[H:G_i]:=[H:H\cap G_i]$ is finite.
\begin{equation*}
c+H\subseteq \bigcup_{i\in I}c_i+G_i\ \ \ \Rightarrow\ \ \ c+H\subseteq \bigcup_{i\in I_0}c_i+G_i,
\end{equation*}
where $I_0=\{i\in I:[H:G_i]<\infty\}$.
\end{lemma}

\section{Special Case: Additive Structure}\label{spcaseadd}
\subsection{The condition $\mathrm{T=T^{\aleph_0}}$}\label{TATA}

Let $M$ be a fixed right $\mathcal R$-module. For brevity we denote $Th(M)$ by $T$. We work with this fixed module throughout this section. If $X,Y\subseteq M^n$, $X,Y\neq\emptyset$, then we use the Minkowski sum notation $X+Y$ to denote the set $\{x+y:x\in X,y\in Y\}$. In case $X=\{x\}$, we use $x+Y$ to denote $X+Y$.

\begin{pro}\label{EC}
The following conditions are equivalent for a module $M$.
\begin{enumerate}
	\item $\mathrm{Inv}(M;A,B)$ is either equal to $1$ or $\infty$ for each $A,B\in\mathcal L_n$ such that $0\in A\cap B$, for each $n\geq 1$,
	\item $M\equiv M\oplus M$,
	\item $M\equiv M^{(\aleph_0)}$.
\end{enumerate}
\end{pro}

\begin{definition}
The theory $T=Th(M)$ is said to satisfy the condition $T=T^{\aleph_0}$ if either (and hence all) of the conditions of proposition \ref{EC} hold.
\end{definition}

We wish to add yet one more condition to the list. The rest of this section is devoted to formulating the condition and deriving its consequences.

We need to introduce some new notation to do this. Let us denote the set of all finite subsets of $\mathcal L_n\setminus\{\emptyset\}$ by $\mathcal P_n$ and the set of all finite antichains in $\mathcal L_n\setminus\{\emptyset\}$ by $\mathcal A_n$. Clearly $\mathcal A_n\subseteq\mathcal P_n$ for each $n\geq 1$. We use lowercase Greek letters to denote elements of $\mathcal A_n$ and $\mathcal P_n$.

\begin{definition}
A definable subset $A$ of $M^n$ will be called $pp$-\textbf{convex} if there is some $\alpha\in\mathcal P_n$ such that $A=\bigcup\alpha$.
\end{definition}

Neumann's lemma (\ref{NL}) takes the following simple form if we add the equivalent conditions of \ref{EC} to our hypotheses.
\begin{cor}\label{NLU}
Suppose $T=T^{\aleph_0}$. If $A\in\mathcal L_n$ and $\mathcal F\subseteq\mathcal L_n$ such that $A\subseteq\bigcup\mathcal F$, then $A\subseteq F$ for at least one $F\in\mathcal F$.
\end{cor}

Under the same hypotheses, we want to show that for every $\alpha\in\mathcal P_n$ the $pp$-convex set $\bigcup\alpha$ is uniquely determined by the antichain $\beta\subseteq\alpha$ of all maximal elements in $\alpha$.

\begin{pro}\label{UNIQUE1}
Suppose that $T=T^{\aleph_0}$ holds. Let $A\subseteq M^n$ be a $pp$-convex set for some $n\geq 1$. Then there is a unique $\beta\in\mathcal{A}_n$ such that $A=\bigcup\beta$.
\end{pro}
\begin{proof}
Let $\alpha_1,\alpha_2\in\mathcal P_n$ be such that $A=\bigcup\alpha_1=\bigcup\alpha_2$. Without loss of generality we may assume $\alpha_1,\alpha_2\in\mathcal{A}_n$. Let $\alpha_1=\{C_1,C_2,\hdots,C_l\}$ and $\alpha_2=\{D_1,D_2,\hdots,D_m\}$.

We have $D_j\subseteq\bigcup_{i=1}^lC_i$ for each $1\leq j\leq m$. Then by \ref{NLU}, we have $D_j\subseteq C_i$ for at least one $i$. By symmetry we also get that each $C_i$ is contained in a $D_j$. Using that both $\alpha_1$ and $\alpha_2$ are antichains with the same union, the proof is complete.
\end{proof}

This proposition shows that under the hypothesis $T=T^{\aleph_0}$ the set of $pp$-convex subsets of $M^n$ is in bijection with $\mathcal A_n$ for each $n\geq 1$. We shall often use this correspondence without mention. For $\alpha\in\mathcal A_n$, we define the \textbf{rank} of the $pp$-convex set $\bigcup\alpha$ to be the integer $|\alpha|$.

The set $\mathcal{A}_n$ can be given the structure of a poset by introducing the relation $\prec_n$ defined by $\beta\prec_n\alpha$ if and only if for each $B\in\beta$, there is some $A\in\alpha$ such that $B\subsetneq A$.

\begin{deflem}\label{UNIQUE2}
Assume that $T=T^{\aleph_0}$. We say that a definable subset $C$ of $M^n$ is a \textbf{cell} if there are $\alpha,\beta\in\mathcal A_n$ with $\beta\prec_n\alpha$ such that $C=\bigcup\alpha\setminus\bigcup\beta$. We denote the set of all cells contained in $M^n$ by $\mathcal C_n$. The antichains $\alpha$ and $\beta$, denoted by $P(C)$ and $N(C)$ respectively, are uniquely determined by the cell $C$. In other words, there is a bijection between the set $\mathcal C_n$ and the set of pairs of antichains strictly related by $\prec_n$. In case $|P(C)|=1$, we say that $C$ is a \textbf{block}. We denote the set of all blocks in $\mathcal C_n$ by $\mathcal B_n$.
\end{deflem}

\begin{proof}
Given any $\alpha,\beta\in\mathcal A_n$ such that $\beta\prec_n\alpha$ and $C=\bigcup\alpha\setminus\bigcup\beta$, the $pp$-convex set $\bigcup(\alpha\cup\beta)$ is determined by $C$. But this set is uniquely determined by the set of maximal elements in $\alpha\cup\beta$ by \ref{UNIQUE1}. Since $\beta\prec_n\alpha$, the required set of maximal elements is precisely $\alpha$. Furthermore, the set $\bigcup\alpha\setminus C=\bigcup\beta$ is $pp$-convex and thus is uniquely determined by $\beta$ by \ref{UNIQUE1} and this finishes the proof.
\end{proof}

We know from \ref{REP} that any definable subset of $M^n$ can be represented as a disjoint union of blocks. So it will be important for us to understand the structure of blocks in detail. A block is always nonempty since any finite union of proper $pp$-subsets cannot cover the given $pp$-set by \ref{UNIQUE1}. For each $B\in\mathcal B_n$, we use the notation $\overline B$ to denote the unique element of $P(B)$.

\begin{theorem}\label{MINK}
Let $M$ be an $\mathcal{R}$-module. Then $Th(M)=Th(M)^{\aleph_0}$ if and only if for each $B\in\mathcal B_n,\ n\geq 1$, we have $B+B-B=\overline B$. Under these conditions, we also get $B-B=\overline{B}$ whenever $\overline{B}$ is a subgroup.
\end{theorem}

\begin{proof}
Assume that $Th(M)=Th(M)^{\aleph_0}$ holds. Let $B\in\mathcal B_n$ be such that $N(B)=\{D_1,D_2,\hdots,D_l\}\prec P(B)=\{A\}$. Let $D=\bigcup N(B)$. We want to show that $B+B-B=A$. But clearly $B\subseteq B+B-B$. So it suffices to show that $D\subseteq B+B-B$.

First assume that $A$ is a subgroup of $M^n$. Let $d\in D$. Since $A\setminus (D-d)\neq\emptyset$, we can choose some $x\in A\setminus (D-d)$ by \ref{NLU}. Then $x+d\in (A+d)\setminus D=A\setminus D=B$, since $A$ is a subgroup. Again choose some $y\in A\setminus((D-d)\cup(D-d-x))$. Then $y+d\in (A+d)\setminus(D\cup(D-x))$ for similar reasons. Thus $y+d, y+x+d\in A\setminus D=B$. Now $d=(d+x)+(d+y)-(d+x+y)\in B+B-B$ and hence the conclusion follows.

In the case when $A$ is a coset of a $pp$-definable subgroup $G$, say $A=a+G$, let $C=D-a$. Then, by the first case, $G=C+C-C$. Now if $d\in A$, then $d-a\in G$. Hence there are $x,y,z\in C$ such that $d-a=x+y-z$. Thus $d=(x+a)+(y+a)-(z+a)\in B+B-B$ and this completes the proof in one direction.

For the converse, suppose that $Th(M)\neq Th(M)^{\aleph_0}$. Then there are two $pp$-definable subgroups $G,H$ of $M^n$ for some $n\geq 1$ such that $H\leq G$ and $1<[G:H]<\infty$. Let $[G:H]=k$ and let $H_1, H_2,\cdots, H_k$ be the distinct cosets of $H$ in $G$. Since $H$ is a $pp$-set, all the cosets $H_i$ are $pp$-sets as well. Now let $B=H_k=G\setminus\bigcup_{i=1}^{k-1}H_i$. Then $B$ is a nonempty block since $k>1$. But, since $B$ is a coset, $B+B-B=B\neq G$ which proves the result in the other direction.

Now we prove the last statement under the hypothesis $Th(M)=Th(M)^{\aleph_0}$. Let $B,A,D$ be as defined in the first paragraph of the proof and assume that $A$ is a subgroup of $M^n$. Given any $a\in A$, we choose $x\in A\setminus (D\cup(D-a))$, which is possible by \ref{NLU}. Then $x,x+a\in B$ and hence $a=(x+a)-x\in (B-B)$. This shows the inclusion $A\subseteq B-B$. We clearly have $(B-B)\subseteq (A-A)$ and $A-A=A$ since $A$ is a subgroup. This completes the proof.
\end{proof}

A map $f:B\rightarrow M^n$ is \textbf{linear} if $f(x+y-z)=f(x)+f(y)-f(z)$ for all $x,y,z\in B$ such that $x+y-z\in B$. We use the previous theorem to show that any linear map on $B$ can be extended uniquely to a linear map on $\overline B$.

\begin{lemma}\label{COLOURINJ}
Suppose that $T=T^{\aleph_0}$ holds. Then for each $n\geq 1$, each $B\in\mathcal B_n$ and each injective linear map $f:B\rightarrowtail M^n$, there exists a unique injective linear extension $\overline{f}:\overline B\rightarrowtail M^n$.
\end{lemma}
\begin{proof}
Let $B=A\setminus\bigcup_{i=1}^mD_i$ be a block and assume that $A$ is a $pp$-definable subgroup. Let $D=\bigcup_{i=1}^mD_i$ and, for each $i$, let $H_i$ denote the subgroup of $A$ whose coset is $D_i$. Let $H=\bigcup_{i=1}^m H_i$. We choose and fix elements $x_i\in B$ sequentially depending on the earlier choices as follows. We choose $x_1\in A\setminus(D\cup H)$ and for each $1<i\leq m$, choose $x_i\in A\setminus(D\cup H\cup \bigcup_{j=1}^{i-1}(D+x_j))$. We can choose $x_i$ at each step by \ref{NLU}. Then we define $\overline{f}(d_i)=f(x_i+d_i)-f(x_i)$ for $d_i\in D_i$ and $\overline{f}(b)=f(b)$ for $b\in B$.

$\overline{f}$ is well-defined: Let $d\in D_i\cap D_j$ for some $j<i$. Then by the choice of $x_i$, $(x_i-x_j)\in B$. Also $x_i,x_i+d,x_j,x_j+d\in B$. Hence $f(x_i-x_j)$ is defined and is equal to both $f(x_i)-f(x_j)$ and $f(x_i+d)-f(x_j+d)$. Hence we see that $f(x_i+d)-f(x_i)=f(x_j+d)-f(x_j)$, which proves that $\overline{f}(d)$ is well-defined for each $d\in D$.

$\overline{f}$ is linear: Let $b\in B$ and $d\in D_i$. Then there are two possibilities namely, $b+d\in B$ or $b+d\in D_j$ for some $j$. In the former case we have $\overline{f}(b+d)=f(b+d)=f(b+x_i+d-x_i)=f(b)+f(x_i+d)-f(x_i)=\overline{f}(b)+\overline{f}(d)$ since $x_i+d,x_i,b\in B$, while in the latter case we have $\overline{f}(b+d)=f(b+d+x_j)-f(x_j)=f(b+d+x_j-x_i+x_i)-f(x_j)=f(b)+f(x_i+d)-f(x_i)+f(x_j)-f(x_j)= f(b)+f(x_i+d)-f(x_i)=\overline{f}(b)+\overline{f}(d)$ since $b,x_i,x_j,x_i+d,x_j+d\in B$.

Let $b\in D_k$ and $d\in D_i$. Then there are two possibilities namely, $b+d\in B$ or $b+d\in D_j$ for some $j$. In the former case we have $\overline{f}(b+d)=f(b+d)=f(b+x_k-x_k+d+x_i-x_i)=
f(b+x_k)-f(x_k)+f(d+x_i)-f(x_i)=\overline{f}(b)+\overline{f}(d)$ since $b+x_k,x_k,x_i,x_i+d\in B$, while in the latter case we have $\overline{f}(b+d)=f(b+d+x_j)-f(x_j)=f(b+x_k-x_k+d+x_i-x_i+x_j)-f(x_j)=
f(b+x_k)-f(x_k)+f(d+x_i)-f(x_i)+f(x_j)-f(x_j)=\overline{f}(b)+\overline{f}(d)$ since $b+x_k,x_k,x_i,x_i+d,x_j\in B$.

In the case when $b,d\in B$ and $b+d\in B$, the linearity of $\overline{f}$ follows from the linearity of $f$. When $b+d\in D_i$, $\overline{f}(b+d)=f(b+d+x_i)-f(x_i)=f(b)+f(d)+f(x_i)-f(x_i)=
\overline{f}(b)+\overline{f}(d)$ since $b,d,x_i\in B$. So we have showed in each case that $\overline{f}$ is linear.

$\overline{f}$ is injective: Without loss we may assume that $\overline f(0)=0$, otherwise we may consider the function $\overline f(-)-\overline f(0)$. Let $a\in A$ be such that $\overline{f}(a)=0$. Then if $a\in B$, then $f(a)=0$ and hence $a=0$ by injectivity of $f$. If $a\in D_i$, then $f(x_i+a)-f(x_i)=0$ and hence $f(x_i+a)=f(x_i)$. But then $x_i+a=x_i$ by injectivity of $f$ since both $x_i+a,x_i\in B$. This again implies that $a=0$.

$\overline{f}$ is unique: Let $h:A\rightarrow M^n$ be any linear injective extension of $f$. Then $h(b)=f(b)=\overline{f}(b)$ for each $b\in B$ and hence, for $d\in D_i$, we have $\overline{f}(d)=f(x_i+d)-f(x_i)=h(x_i+d)-h(x_i)=h(d)$ since $x_i+d,x_i\in B$ and $h$ is linear.

If $A$ is a nontrivial coset of some $pp$-definable subgroup $G$ of $M^n$, $D\subsetneq A$, $B=A\setminus D$ and we are given some linear map $f:B\rightarrowtail M^n$, we choose and fix some $b\in B$. Then clearly $A-b=G$ and we take $C=B-b$. Define $g:C\rightarrow M^n$ by setting $g(c)=f(c+b)-f(b)$. Now whenever $x,y\in C$ such that $x+y\in C$, we have $g(x+y)=f(x+y+b)-f(b)=f((x+b)-b+(y+b))-f(b)= f(x+b)-f(b)+f(y+b)-f(b)=g(x)+g(y)$ since $x+b,y+b,b\in B$ and $f$ is linear. Hence $g$ is linear on $C$. Therefore by the subgroup case, we have a unique linear injective extension of $g$ to $G$, say $\overline{g}$. Then we define $\overline{f}:A\rightarrow M^n$ by setting $\overline{f}(a)=\overline{g}(a-b)+f(b)$. It can be easily seen that $\overline{f}$ is indeed an extension of $f$. The uniqueness, linearity and injectivity of $\overline{f}$ follows from the uniqueness of $\overline{g}$. This argument completes the proof of this case and hence that of the lemma.
\end{proof}

\subsection{Local characteristics}\label{LC}

We fix some $n\geq 1$ and drop all the subscripts $n$. We also assume hereafter that $Th(M)=Th(M)^{\aleph_0}$ holds for some fixed right $\mathcal R$-module $M$. For brevity, we denote the sets $\mathcal L\setminus\{\emptyset\},\mathcal A\setminus\{\emptyset\},\hdots$ by $\mathcal L^*,\mathcal A^*,\hdots$ respectively.

\begin{definition}
Let $\mathcal D$ be a finite subset of $\mathcal L^*$. The smallest sub-meet-semilattice of $\mathcal L$ containing $\mathcal D$ will be called the $pp$-nest (or simply nest) corresponding to $\mathcal D$ and will be denoted by $\hat{\mathcal D}$. Note that $\hat{\mathcal D}$ is finite. In general, any finite sub-meet-semilattice of $\mathcal L$ will also be referred to as a $pp$-nest.
\end{definition}

\begin{definition}
For each finite subset $\mathcal F$ of $\mathcal L^*$ and $F\in\mathcal F$, we define the \textbf{$\mathcal F$-core} of $F$ to be the block $\mathrm{Core}_\mathcal F(F):=F\setminus\bigcup\{G: G\in\mathcal F, G\cap F\subsetneq F\}$.
\end{definition}

Let $D\subseteq M^n$ be definable. Then $D=\bigsqcup_{i=1}^m B_i$ for some $B_i\in\mathcal B$ by \ref{REP}. We say that $\mathcal D$ is the nest corresponding to this partition of $D$ if it is the nest corresponding to the finite family $\bigcup_{i=1}^m(P(B_i)\cup N(B_i))$. Every definable set can be partitioned canonically given a suitable nest, which is the content of the following lemma whose proof is omitted.

\begin{deflem}\label{CH1}
Suppose $D\subseteq M^n$ is definable and $\mathcal D$ is the nest corresponding to a given partition $D=\bigsqcup_{i=1}^m B_i$. For every nonempty $F\in\mathcal D$, $\mathrm{Core}_\mathcal D(F)\cap~D\neq\emptyset$ if and only if $\mathrm{Core}_\mathcal D(F)\subseteq D$. We define the \textbf{characteristic function} of the nest $\mathcal D$, $\delta_\mathcal D:\mathcal D\rightarrow \{0,1\}$, by $\delta_\mathcal D(F)=1$ if and only if $F\neq\emptyset$ and $\mathrm{Core}_\mathcal D(F)\subseteq D$. We denote the sets $\delta_\mathcal D^{-1}(1)$ and $\delta_\mathcal D^{-1}(0)$ by $\mathcal D^+$ and $\mathcal D^-$ respectively. Then $D=\bigcup_{F\in\mathcal D^+}\mathrm{Core}_\mathcal D(F)$.
\end{deflem}

\begin{illust}
If $B$ is a block with $P(B)=A$ and $N(B)=\{D_1,D_2\}$ such that $D_1\cap D_2\neq\emptyset$, then $\mathcal D=\{A,D_1,D_2,D_1\cap D_2\}$ is the corresponding nest. Clearly $B=\mathrm{Core}_\mathcal D(A)$ and hence $\mathcal D^+=\{A\}$.
\end{illust}

We will sometimes use another family of characteristic functions stated in the following definition.
\begin{definition}\label{CH2}
Given any $C\in\mathcal C$, we define the \textbf{characteristic function} of the cell $C$, $\delta(C):\mathcal L^*\rightarrow\{0,1\}$, as $\delta(C)(P)=1$ if $P\subseteq C$ and $\delta(C)(P)=0$ otherwise, for each $P\in\mathcal{L}^*$. When $P=\{a\}$, we write the expression $\delta(C)(a)$ instead of $\delta(C)(\{a\})$.
\end{definition}

The set $\mathcal A$ of antichains is ordered by the relation $\prec$ but can also be considered as a poset with respect to the natural inclusion ordering on the set of all $pp$-convex sets. For $\alpha,\beta\in\mathcal{A}$, we define $\alpha\wedge\beta$ to be the antichain corresponding (in the sense of \ref{UNIQUE1}) to $(\bigcup\alpha)\cap(\bigcup\beta)$ and $\alpha\vee\beta$ to be the antichain corresponding to $(\bigcup\alpha)\cup(\bigcup\beta)$. Since the intersection and union of two $pp$-convex sets are again $pp$-convex, the binary operations $\wedge,\vee:\mathcal A\times\mathcal A\rightarrow\mathcal A$ are well-defined. It can be easily seen that $\mathcal{A}$ is a distributive lattice with respect to these operations.

We want to understand the structure of any definable set ``locally'' in a neighbourhood of a point in $M^n$. The following lemma defines a class of sub-lattices of $\mathcal A$ which provides the necessary framework to define the concept of localization. The proof is an easy verification of an adjunction and is not given here.

\begin{deflem}
Fix some $a\in M^n$. Let $\mathcal{L}_a:=\{A\in\mathcal{L}:a\in A\}$ and $\mathcal{A}_a$ denote the set of all antichains in the meet-semilattice $\mathcal{L}_a$. Then $\mathcal{A}_a$ is a sub-lattice of $\mathcal{A}$. We denote the inclusion $\mathcal{A}_a\rightarrow\mathcal{A}$ by $\mathcal{I}_a$.
We also consider the map $\mathcal{N}_a:\mathcal{A}\rightarrow\mathcal{A}_a$ defined by $\alpha\mapsto \alpha\cap\mathcal{L}_a$. We call the antichain $\mathcal{N}_a(\alpha)$ the \textbf{localization} of $\alpha$ at $a$. Then $\mathcal{N}_a$ is a right adjoint to $\mathcal{I}_a$ if we consider the posets $\mathcal A$ and $\mathcal A_a$ as categories in the usual way, and the composite $\mathcal{N}_a\circ\mathcal{I}_a$ is the identity on $\mathcal{A}_a$. This in particular means that $\mathcal{A}_a$ is a reflective subcategory of $\mathcal{A}$. Furthermore, the map $\mathcal{N}_a$ not only preserves the meets of antichains, being a right adjoint, but it also preserves the joins of antichains.
\end{deflem}

Fix some $a\in M^n$. Let us denote the set of all finite subsets of $\mathcal{L}_a$ by $\mathcal P_a$ and let $\alpha\in\mathcal{P}_a$. We construct a simplicial complex $\mathcal{K}^a(\alpha)$ which determines the ``geometry'' of the intersection of elements of $\alpha$ around $a$. This construction is similar to the construction of the nerve of an open cover, except for the meaning of the ``triviality'' of the intersection. We know that a $pp$-set is finite if and only if it has at most $1$ element. We also know that $\bigcap\alpha\supseteq\{a\}$.

\begin{definition}\label{SIMPCOMP}
We associate an abstract simplicial complex $\mathcal K^a(\alpha)$ to each $\alpha\in\mathcal P_a$ by taking the vertex set $\mathcal V(\mathcal K^a(\alpha)):=\alpha\setminus\{a\}$. We say that a nonempty set $\beta\subseteq\alpha$ is a face of $\mathcal{K}^a(\alpha)$ if and only if $\bigcap\beta$ is infinite (i.e., strictly contains $a$). If the only element of $\alpha$ is $\{a\}$ or if $\alpha=\emptyset$, then we set $\mathcal{K}^a(\alpha)=\emptyset$, the empty complex.
\end{definition}

\begin{illust}
Consider the real vector space $\mathbb R_{\mathbb R}$. The theory of this vector space satisfies the condition $\rm T=T^{\aleph_0}$. We consider subsets of $\mathbb R^3$. If $\alpha$ denotes the antichain corresponding to the union of $3$ coordinate planes and $a$ is the origin, then $\mathcal K^a(\alpha)$ is a copy of $\partial\Delta^2$. The $2$-dimensional face of $\Delta^2$ is absent since the intersection of the coordinate planes does not contain the origin properly.
\end{illust}

Since $\beta_1\subseteq\beta_2\ \Rightarrow\ \bigcap\beta_2\subseteq\bigcap\beta_1$, $\mathcal{K}^a(\alpha)$ is indeed a simplicial complex. We tend to drop the superscript $a$ when it is clear from the context. To extend this definition to arbitrary elements of $\mathcal P$, we extend the notion of localization operator (at $a$) to $\mathcal{P}$ by setting $\mathcal{N}_a(\alpha)=\alpha\cap\mathcal{L}_a$ for each $\alpha\in\mathcal P$. Now we are ready to define a family of numerical invariants for convex subsets of $M^n$, which we call ``local characteristics''.

\begin{definition}
We define the function $\kappa_a:\mathcal{P}\rightarrow\mathbb{Z}$ by setting $\kappa_a(\alpha):= \chi(\mathcal{K}(\mathcal{N}_a(\alpha)))- \delta(\alpha)(a)$, where $\chi(\mathcal{K})$ denotes the Euler characteristic of the complex $\mathcal{K}$ as defined in \ref{Euler} and $\delta(\alpha)$ is the characteristic function of the set $\bigcup\alpha$ as defined in \ref{CH2}. The value $\kappa_a(\alpha)$ will be called the \textbf{local characteristic} of the antichain $\alpha$ at $a$.
\end{definition}

If we view the local characteristic $\kappa_a(\alpha)$ as a function of $a$ for a fixed antichain $\alpha$, the correction term $\delta(\alpha)(a)$ makes sure that $\kappa_a(\alpha)=0$ for all but finitely many values of $a$. This fact will be useful in the next section.

We want to show that the local characteristic satisfies the inclusion-exclusion principle for antichains.
\begin{theorem}\label{t1}
For $\alpha,\beta\in\mathcal{A}$, we have $\kappa_a(\alpha\vee\beta)+\kappa_a(\alpha\wedge\beta)=\kappa_a(\alpha)+\kappa_a(\beta)$.
\end{theorem}

The rest of this section is devoted to the proof of this theorem. First we observe that it is sufficient to prove this theorem for $\alpha,\beta\in\mathcal{A}_a$. We also observe that it is sufficient to prove this theorem in the case when $\kappa_a$ is replaced by the function $\chi(\mathcal K(-))$ because $\kappa_a(\alpha)=\chi(\mathcal{K}(\alpha))-1$ whenever $a\in\bigcup\alpha$ and the cases when either $a\notin\bigcup\alpha$ or $a\notin\bigcup\beta$ are trivial. We write $\kappa_a$ as $\kappa$ for simplicity of notation.

The following proposition is the first step in this direction, which states that $\kappa(\alpha)$ is actually determined by the $pp$-convex set $\bigcup\alpha$.

\begin{pro}\label{p1}
Let $\alpha\in\mathcal{A}_a$ and $\beta\in\mathcal{P}_a$. If $\bigcup\alpha=\bigcup\beta$, then $\kappa(\alpha)=\kappa(\beta)$.
\end{pro}
\begin{proof}
It is clear that $\beta\supseteq\alpha$ since $\beta$ is finite. Hence $\mathcal{K}(\alpha)$ is a full sub-complex of $\mathcal{K}(\beta)$ (i.e. if $\beta'\in \mathcal K(\beta)$ and $\beta'\subseteq\alpha$, then $\beta'\in\mathcal K(\alpha)$). We can also assume that $\{a\}\notin\beta$. Note that every element $\beta\setminus\alpha$ is properly contained in at least one element of $\alpha$. We use induction on the size of $\beta\setminus\alpha$ to prove this result.

If $\beta\setminus\alpha=\emptyset$, then the conclusion is trivially true. For the inductive case, suppose $\alpha\subseteq\beta'\subsetneq\beta$ and the result has been proved for $\beta'$. Let $B\in\beta\setminus\beta'$. Since $\alpha$ is the set of maximal elements of $\beta$, there is some $A\in\alpha$ such that $A\supsetneq B$.

Consider the complex $\mathcal{K}_1=\{F\in\mathcal{K}(\beta'):(F\cup\{B\})\in\mathcal{K}(\beta'\cup\{B\})\}$ as a full sub-complex of $\mathcal{K}(\beta')$. Observe that whenever $B\in F\in\mathcal{K}(\beta'\cup\{B\})$, we have $(F\cup\{A\})\setminus\{B\}\in\mathcal{K}(\beta')$. As a consequence, $\mathcal K_1=\mathrm{Cone}(\mathcal K(\beta'\setminus\{A\}))$ where the apex of the cone is $A$. In particular, $\mathcal K_1$ is contractible.

Also note that $\mathcal{K}(\beta'\cup\{B\})=\mathcal{K}(\beta')\cup\mathrm{Cone}(\mathcal{K}_1)$, where the apex of the cone is $B$. Now we compare the pair $\mathcal K_1\subseteq\mathcal{K}(\beta')$ with another pair $\mathrm{Cone}(\mathcal K_1)\subseteq\mathcal K(\beta'\cup\{B\})$ of simplicial complexes. Observe the set equality $\mathcal{K}(\beta')\setminus\mathcal K_1=\mathcal K(\beta'\cup\{B\})\setminus \mathrm{Cone}(\mathcal K_1)$. Also both $\mathcal K_1$ and $\mathrm{Cone}(\mathcal K_1)$ are contractible. Thus we conclude that $\mathcal{K}(\beta'\cup\{B\})$ and $\mathcal{K}(\beta')$ are homotopy equivalent. Finally, an application of \ref{EULHTPY} completes the proof.
\end{proof}

Note that this result is very helpful for the computation of local characteristics as we get the equalities $\kappa(\alpha\vee\beta)=\kappa(\alpha\cup\beta)$ and $\kappa(\alpha\wedge\beta)=\kappa(\alpha\circ\beta)$ for all $\alpha,\beta\in\mathcal{A}_a$, where $\alpha\circ\beta=\{A\cap B: A\in\alpha, B\in\beta\}$. The vertices of $\mathcal{K}(\alpha\circ\beta)$ will be denoted by the elements from $\alpha\times\beta$.

We use induction twice, first on $|\beta|$ and then on $|\alpha|$, to prove the main theorem of this section. The following lemma is the first step of this induction.

\begin{lemma}
For $\alpha,\beta\in\mathcal{A}_a$ and $\left|\alpha\right|\leq 1$, we have $\kappa(\alpha\vee\beta)+\kappa(\alpha\wedge\beta)=\kappa(\alpha)+\kappa(\beta)$.
\end{lemma}
\begin{proof}
The cases $\left|\alpha\right|=0$ and $\alpha=\{\{a\}\}$ are trivial. So we assume that $\alpha=\{A\}$ where $A$ is infinite. We can make similar non-triviality assumptions on $\beta$, namely there is at least one element in $\beta$ and all the elements of $\beta$ are infinite.

There are only two possible cases when $|\beta|=1$ and the conclusion holds true in both these cases. For example when $\beta=\{B\}$ and $A\cap B=\{a\}$, we have $\mathcal K(\alpha)\cong\mathcal K(\beta)\cong\Delta^0$, $\mathcal K(\alpha\circ\beta)$ is empty and $\mathcal K(\alpha\cup\beta)$ is disjoint union of two copies of $\Delta^0$. Hence the identity in the statement of the lemma takes the form $1+(-1)=0+0$.

Suppose for the inductive case that the result is true for $\beta$ i.e. $\kappa(\alpha\vee\beta)+\kappa(\alpha\wedge\beta)=\kappa(\alpha)+\kappa(\beta)$ holds. We want to show that the result holds for $\beta\cup\{B\}$ i.e. $\kappa(\alpha\vee(\beta\cup\{B\}))+\kappa(\alpha\wedge(\beta\cup\{B\}))=\kappa(\alpha)+\kappa(\beta\cup\{B\})$.

We introduce some superscript and subscript notations to denote new simplicial complexes obtained from the original. The following list describes them and also explains the rules to handle two or more scripts at a time.
\begin{itemize}
    \item Let $\mathcal K_0$ denote the complex $\mathcal K(\alpha)$, i.e. the complex consisting of only one vertex and $\mathcal{K}$ denote the complex $\mathcal{K}(\beta)$.
    \item Let $\mathcal{K}^S$ denote the complex $\mathcal{K}(\beta\cup S)$ for any finite $S\subseteq\mathcal{L}_a$ which contains only infinite elements. Also, $\mathcal K^{A,B}$ is a short hand for $\mathcal K^{\{A,B\}}$.
	\item Whenever $C$ is a vertex of $\mathcal{Q}$, the notation $\mathcal{Q}_C$ denotes the sub-complex $\{F\in\mathcal{Q}:C\notin F, F\cup\{C\}\in\mathcal{Q}\}$ of $\mathcal{Q}$.
	\item If $\mathcal Q=\mathcal K(\gamma)$ for some antichain $\gamma$ and $A\notin\gamma$, then the notation $^A\mathcal{Q}$ denotes the complex $\mathcal{K}(\{A\}\circ\gamma)$.
	\item The notation $^C\mathcal K^S_B$ means $^C((\mathcal K^S)_B)$. This describes the order of the scripts.
    \item The Euler characteristic of $^C\mathcal K^S_B$ will be denoted by $^C\chi^S_B$.
\end{itemize}

Using this notation, the inductive hypothesis is
\begin{equation}\label{1}\chi^A+\,^A\chi=\chi_0+\chi\end{equation} and our claim is \begin{equation}\label{2}\chi^{B,A}+\,^A\chi^B=\chi_0+\chi^B.\end{equation}

\textbf{Case I: $(A\cap B)=\{a\}.$} In this case, the faces of $\mathcal{K}^{A,B}$ not present in $\mathcal{K}^A$ are the faces of $\mathcal{K}^B$. Hence $b_n(\mathcal{K})-b_n(\mathcal{K}^B)=b_n(\mathcal{K}^A)-b_n(\mathcal{K}^{A,B})$ for all $n\geq 0$, where $b_n$ denotes the $n^{th}$ Betti number. Hence we get
\begin{equation*}\chi^{B,A}-\chi^A=\chi^B-\chi\end{equation*}
Also note that the hypothesis $(A\cap B)=\{a\}$ yields $H_*(^A\mathcal{K})=H_*(^A\mathcal{K}^B)$ since only infinite elements matter for the computations. It follows that equation (\ref{2}) holds in this case.

\textbf{Case II: $A\cap B\supsetneq\{a\}.$} Note that whenever $C$ is not a vertex of $\mathcal Q$, we have $\mathcal Q^C_C\subseteq\mathcal Q$ and $\mathcal Q\cup\mathrm{Cone}(\mathcal Q^C_C)=\mathcal Q^C$, where the apex of the cone is $C$. Hence corollary \ref{HMLGCONE} can be restated in this notation as the following identity. \begin{equation}\label{3}\chi(\mathcal{Q})+1=\chi(\mathcal{Q}^C)+\chi(\mathcal{Q}^C_C)\end{equation}
As particular cases of (\ref{3}), we get the following equalities.
\begin{equation}\label{4}\chi+1=\chi^B+\chi^B_B.\end{equation} \begin{equation}\label{5}\chi^A+1=\chi^{A,B}+\chi^{A,B}_B\end{equation} \begin{equation}\label{6}\chi^B_B+1=\chi^{A,B}_B+\chi^{A,B}_{A,B}\end{equation}
It can be checked that $\mathcal{K}^{A,B}_{A,B}\cong\ ^A\mathcal{K}^B_B$ via the map $F\mapsto \{\{C,A\}: C\in F\}$. This gives us the following equation.
\begin{equation}\label{8}\chi^{A,B}_{A,B}=\ ^A\chi^B_B\end{equation}
If we combine equations (\ref{1}),(\ref{4}),(\ref{5}),(\ref{6}) and (\ref{8}), it remains to prove the following to get equation (\ref{2}) in the claim.
\begin{equation}\label{7}^A\chi+1=\,^A\chi^B+\,^A\chi^B_B\end{equation}
Observe that the natural inclusion maps $i_1:\mathcal{K}_0\rightarrow\mathcal{K}^A$ and $i_2:\mathcal{K}\rightarrow\mathcal{K}^A$ are inclusions of sub-complexes and their images are disjoint. Furthermore, the set theoretic map $g:\mathcal{K}^A\setminus(Im(i_1)\sqcup Im(i_2))\rightarrow\,^A\mathcal{K}$ defined by $F\mapsto\{\sigma\subseteq F:A\in\sigma,|\sigma|=2\}$ is a bijection. Now consider the composition $^A\mathcal{K}^B\cong\mathcal{K}^{A,B}\setminus(i_1(\mathcal{K}_0)\sqcup i_2(\mathcal{K}^B))\xrightarrow{\pi_B}\mathcal{K}^A\setminus(i_1(\mathcal{K}_0)\sqcup i_2(\mathcal{K}))\cong\, ^A\mathcal{K}$, where $\pi_B(F)=F\setminus\{B\}$. The union of images (under this composition of maps) of those faces in $^A\mathcal K^B$ which contain $A\cap B$ is the sub-complex $^A\mathcal{K}^B_B$ of $^A\mathcal{K}$. Hence $(^A\mathcal{K}\cup \mathrm{Cone}(^A\mathcal{K}^B_B))\cong\,^A\mathcal{K}^B$, where the apex of the cone is $\{A,B\}$. An application of \ref{HMLGCONE} gives the required identity in equation (\ref{7}).
\end{proof}

We use definition \ref{Euler} of Euler characteristic to prove the second step in the proof of the main theorem since we do not have a proof using homological techniques. In this step, we allow the size of $\beta$ to be an arbitrary fixed positive integer and we use induction on the size of $\alpha$. The lemma just proved is the base case for this induction. Let $A$ be a new element of $\mathcal L_a$ to be added to $\alpha$ and assume the result is true for $\alpha$. Again we may assume that $A$ is infinite.

We construct the complex $\mathcal{K}(\alpha\cup\beta\cup\{A\})$ in steps starting with the complex $\mathcal{K}(\alpha\cup\beta)$ and the conclusion of the theorem holds for the latter by the inductive hypothesis. We do this in such a way that at each step $\mathcal{K}_1$ of the construction, the following identity is satisfied.
\begin{equation}\label{count}
\chi(\mathcal K(\alpha\cup\{A\})\cap\mathcal K_1)+\chi(\mathcal K(\beta))=\chi(\mathcal K(\alpha\cup\{A\}\cup\beta)\cap\mathcal K_1)+\chi(\mathcal K((\alpha\cup\{A\})\circ\beta)\cap\mathcal K_1)
\end{equation}
In this expression, $\mathcal K((\alpha\cup\{A\})\circ\beta)\cap\mathcal K_1$ denotes the subcomplex of $\mathcal K((\alpha\cup\{A\})\circ\beta)$ whose faces are appropriate projections of the faces of $\mathcal K_1$.

For the first step, we construct all the elements in $\mathcal{K}(\alpha\cup\{A\})$ not in $\mathcal K(\alpha)$. Let $\mathcal K_1'$ denote the resulting complex. No new faces of the complex $\mathcal K((\alpha\cup\{A\})\circ\beta)$ are constructed in this process. Hence, for each $n\geq 0$, we have
\begin{equation*}
v_n(\mathcal K_1')-v_n(\mathcal K(\alpha\cup\beta\cup\{A\}))=v_n(\mathcal K_1'\cap\mathcal{K}(\alpha\cup\{A\}))-v_n(\mathcal{K}(\alpha\cup\{A\})),
\end{equation*}
where $v_n(\mathcal Q)$ denotes the number $n$-dimensional faces of $\mathcal Q$. Hence equation (\ref{count}) is satisfied for $\mathcal K_1'$.

For the second step, we further construct all the faces corresponding to $\{A\}\circ\beta$. The conclusion in this case follows from the previous lemma.

Finally we construct the faces containing $A$ and intersecting both $\alpha$ and $\beta$. We construct a face $F$ of size $m+k$ whenever all the proper sub-faces of $F$ have already been constructed, where $F\cap(\alpha\cup\{A\})$ and $F\cap\beta$ have sizes $m$ and $k$ respectively. It is clear that $m,k\geq 1$.

Let the sub-complex of $\mathcal{K}(\alpha\cup\beta\cup\{A\})$ consisting of the already constructed faces be denoted by $\mathcal{K}$. We assume, for induction, that equation (\ref{count}) is true for $\mathcal K$. Let $g(F')=\{\sigma\subseteq F':|\sigma\cap(\alpha\cup\{A\})|=1,|\sigma\cap\beta|=1\}$ for $F'\in\mathcal{K}$. Let $\mathcal{K}_3=\bigcup_{F'\subsetneq F}g(F')$ and $\mathcal{K}_2=\bigcup_{F'\in\mathcal{K}}g(F')$. Note the inclusions $\mathcal K\subseteq\mathcal K_2\subseteq\mathcal{K}((\alpha\cup\{A\})\circ\beta)$.

The construction of the face $F$ changes $\chi(\mathcal{K})$ by $(-1)^{\mathrm{dim} F}=(-1)^{m+k-1}$, while the numbers $\chi(\mathcal{K}(\alpha\cup\{A\}))$ and $\chi(\mathcal{K}(\beta))$ are unaltered.

We calculate the change in the value of $\chi(\mathcal{K}_3)$ after the addition of $F$. The complex $g(F)$ is contractible. Hence its Euler characteristic is equal to $1$ by \ref{EULHTPY}.

Let $w_n$ denote the number of $n$-dimensional faces of $\mathcal{K}_3$. Recall that $\mathcal V(\mathcal K_3)= (\alpha\cup\{A\})\times\beta$. If $\mathrm{dim} F'=n+2$ for some $F'\in\mathcal K(\alpha\cup\beta\cup\{A\})$ such that $|F'\cap(\alpha\cup\{A\})|\geq 1,|F'\cap\beta|\geq 1$, then $\mathrm{dim}(g(F'))=n$. Therefore to calculate $w_n$, we choose total $n+2$ vertices from $F$, making sure that we choose at least one vertex from both $\alpha\cup\{A\}$ and $\beta$. Hence $w_n=\Sigma_{j=1}^{n+1}\binom{m}{j}\binom{k}{n+2-j}$. This number can be easily shown to be equal to $\binom{m+k}{n+2}-\binom{m}{n+2}-\binom{k}{n+2}$.

The maximum dimension of the face of $\mathcal{K}_3$ is equal to $m+k-3$. Hence the required change in the value of $\chi(\mathcal{K}_3)$ is $1-\Sigma_{n=0}^{m+k-3}(-1)^n[\binom{m+k}{n+2}-\binom{m}{n+2}-\binom{k}{n+2}]$.

To obtain equation (\ref{count}) for $\mathcal K\cup\{F\}$, we need to show that there is no change in the value of $\chi(\mathcal{K})+\chi(\mathcal K_3)$ after construction of $F$.

But we know that $\Sigma_{n=0}^{m+k}(-1)^n[\binom{m+k}{n}-\binom{m}{n}-\binom{k}{n}]=0$ since each of the three alternating sums is zero. This equation rearranges to give the required cancelation equation and completes the proof.

\subsection{Global characteristic}\label{GCDS}

Let the function $\kappa:\mathcal{A}\times M^n\rightarrow\mathbb{Z}$ be defined by $\kappa(\alpha,a)=\kappa_a(\alpha)$. Suppose $\alpha$ is a singleton. If $\bigcup\alpha$ is infinite, then $\kappa(\alpha,-)$ is the constant $0$ function and if $\alpha=\{a\}$, then $\kappa(\alpha,b)=0$ for all $b\neq a$ and $\kappa(\alpha,a)=-1$. For arbitrary $\alpha\in\mathcal{A}$, if $a\notin\bigcup\alpha$, then $\kappa(\alpha,a)=0$.

\begin{definitions}
For $\alpha\in\mathcal{A}$, we define the \textbf{set of singular points} of $\alpha$ to be the set $\mathrm{Sing}(\alpha):=\{a\in M^n: \kappa(\alpha,a)\neq 0\}$. $\mathrm{Sing}(\alpha)$ is always finite since all the singular points appear as singletons in the nest corresponding to $\alpha$. Using finiteness of $\mathrm{Sing}(\alpha)$, we define the \textbf{global characteristic} of $\alpha$ to be the sum $\Lambda(\alpha):=-\Sigma_{a\in M^n} \kappa(\alpha,a)$, which in fact is equal to the finite sum $\Lambda(\alpha)=-\Sigma_{a\in \mathrm{Sing}(\alpha)}\kappa(\alpha,a)$.
\end{definitions}

Fix some $a\in M^n$. Let $\alpha,\beta\in\mathcal{A}$ be such that $\beta\prec\alpha$. Then either $\mathcal{N}_a(\alpha)=\mathcal{N}_a(\beta)=\emptyset$ or $\mathcal{N}_a(\beta)\prec\mathcal{N}_a(\alpha)$. If $C:=\bigcup\alpha\setminus\bigcup\beta$ is a cell, we define the homology $H_*(C)$ to be the relative homology $H_*(\mathcal{K}(\mathcal{N}_a(\alpha\cup\beta));\mathcal{K}(\mathcal{N}_a(\beta)))$. In particular, the alternating sum of the Betti numbers of $H_*(C)$, denoted by $\chi_a(C)$, is equal to the difference $\chi(\mathcal{K}(\mathcal{N}_a(\alpha)))-\chi(\mathcal{K}(\mathcal{N}_a(\beta)))$ by \ref{p1} and \ref{LONGEXACT}. We also have the equation $\delta(C)=\delta(\alpha)-\delta(\beta)$. Hence if we define the local characteristic of $C$ as $\kappa_a(C):=\chi_a(C)-\delta(C)(a)$, we get the identity $\kappa_a(C)=\kappa_a(P(C))-\kappa_a(N(C))$. We define the extension of the function $\kappa$ to include all cells by setting $\kappa(C,a):=\kappa_a(C)$ for $a\in M^n,C\in\mathcal{C}$.

\begin{definitions}
We define the set of singular points $\mathrm{Sing}(C)$ for $C\in\mathcal{C}$ analogously by setting $\mathrm{Sing}(C):=\{a\in M^n:\kappa_a(C)\neq 0\}$. This set is finite since $\mathrm{Sing}(C)\subseteq \mathrm{Sing}(P(C))\cup \mathrm{Sing}(N(C))$. We also extend the definition of global characteristic for cells by setting $\Lambda(C):=-\Sigma_{a\in M^n} \kappa(C,a)$.
\end{definitions}

It is immediate that $\Lambda(C)=\Lambda(P(C))-\Lambda(N(C))$ for every $C\in\mathcal{C}$.

The main aim of this section is to prove that the global characteristic is additive in the following sense.
\begin{theorem}\label{t2}
If $\{B_i:1\leq i\leq l\}, \{B_j':1\leq j\leq m\}$ are two finite families of pairwise disjoint blocks such that $\bigsqcup_{i=1}^l B_i=\bigsqcup_{j=1}^m B_j'$, then $\Sigma_{i=1}^l\Lambda(B_i)=\Sigma_{j=1}^m\Lambda(B_j')$.
\end{theorem}

The proof of this theorem follows at once from the following local version.
\begin{lemma}\label{l1}
If $a\in M^n$ and $\{B_i:1\leq i\leq l\}, \{B_j':1\leq j\leq m\}$ are two finite families of pairwise disjoint blocks such that $\bigsqcup_{i=1}^l B_i=\bigsqcup_{j=1}^m B_j'$, then $\Sigma_{i=1}^l\kappa_a(B_i)=\Sigma_{j=1}^m\kappa_a(B_j')$.
\end{lemma}
\begin{proof}
It will be sufficient to show that both these numbers are equal to the sum $\Sigma_{B\in\mathcal{F}}\,\kappa_a(B)$ where $\mathcal{F}$ is any finite family of blocks finer than both the given families. We can in particular choose a finite $pp$-nest $\mathcal{D}$ containing all the elements in $\bigcup_{i=1}^l(P(B_i)\cup N(B_i))\cup\bigcup_{j=1}^m(P(B_j)\cup N(B_j))$ and set $\mathcal{F}=\{\mathrm{Core}_\mathcal{D}(D):D\in\mathcal{D}^+\}$. This involves partitioning every $B_i$ and $B_j'$ into smaller blocks of the form $\mathrm{Core}_\mathcal D(D)$ for $D\in\mathcal D^+$.

Thus it will be sufficient to show that if $\mathcal{F}$ is a finite family of blocks corresponding to cores of a $pp$-nest $\mathcal{D}$ such that $B=\bigcup \mathcal{F}\in\mathcal{B}$, then $\kappa_a(B)=\Sigma_{F\in\mathcal{F}}\kappa_a(F)$. Consider the sub-poset $\mathcal{H}$ of $\mathcal{L}$ containing all the elements of $\bigcup_{F\in\mathcal{F}}(P(F)\cup N(F))$. Then we construct the antichains $\{\alpha_s\}_{s\geq 0}$ in such a way that $\alpha_s$ is the set of all minimal elements of $\mathcal{H}\setminus\bigcup_{0\leq t<s}\alpha_t$. Then this process stops, say $\alpha_v$ is $P(B)$. Then we have a chain of antichains $\alpha_0\prec\alpha_1\prec\cdots\prec\alpha_v$. Now $\kappa_a(B)=\kappa_a(\alpha_v)-\kappa_a(\alpha_0)= \Sigma_{t=1}^v\kappa_a(\alpha_t)-\kappa_a(\alpha_{t-1})$. In other words, if $C_t$ denotes the cell $\bigcup\alpha_t\setminus\bigcup\alpha_{t-1}$ for $1\leq t\leq v$, then $\kappa_a(B)=\Sigma_{t=1}^v\kappa_a(C_t)$.

Now it remains to show that for each $1\leq t\leq v$, $\kappa_a(C_t)=\Sigma_{F\in\alpha_t} \kappa_a(\mathrm{Core}_\mathcal{D}(F))$. This follows from the following proposition by first choosing $A_j$ to consist of elements of $\alpha_t$ and then choosing $A_j$ to consist of elements of $\alpha_{t-1}$. Then by our construction of the chain and the definition of $\kappa_a(C_t)$, we get the required result.
\end{proof}

\begin{pro}\label{pr2}
For any $\alpha_j\in\mathcal{A},A_j=\bigcup\alpha_j,j\in[k]=\{1,2,\hdots,k\}$ where $k\geq 2$, we have $\kappa_a(\bigcup_{j\in[k]}A_j)=\Sigma_{S\subseteq[k],S\neq\emptyset}\kappa_a(\bigcap_{s\in S}A_s\setminus\bigcup_{t\notin S}A_t)$.
\end{pro}
\begin{proof}
We observe that all the arguments on the right hand side of the above expression are cells or possibly empty sets and they form a partition of the cell in the argument of the left hand side. Then we restate theorem \ref{t1} as $\kappa_a((\bigcup\alpha)\cup(\bigcup\beta))= \kappa_a((\bigcup\alpha)\setminus(\bigcup\beta))+ \kappa_a((\bigcup\beta)\setminus(\bigcup\alpha))+ \kappa_a((\bigcup\alpha)\cap(\bigcup\beta))$. Since the set of $pp$-convex sets is closed under taking unions and intersections, a simple induction proves the proposition with \ref{t1} being the base case.
\end{proof}

Theorem \ref{t2} allows us to define the global characteristic for arbitrary definable sets.
\begin{definition}
Let $D\subseteq M^n$ be definable. Then we define the \textbf{global characteristic} $\Lambda(D)$ as the sum of global characteristics of any finite family of blocks partitioning $D$.
\end{definition}

\subsection{Preservation of global characteristics}\label{PTGC}

The aim of this section is to show that the global characteristic is preserved under definable isomorphisms.

\begin{theorem}\label{t3}
Suppose $D\in \mathrm{Def}(M^n)$ and $f:D\rightarrow M^n$ is a definable injection. Then $\Lambda(D)=\Lambda(f(D))$.
\end{theorem}
\begin{proof}
We first prove the local version which states that for any $a\in M^n$ and $B\in\mathcal{B}$ if $g:B\rightarrow M$ is a $pp$-definable injection, then $\kappa_a(B)=\kappa_{g(a)}(g(B))$. We observe that $\delta(B)(a)=\delta(g(B))(g(a))$. Lemma \ref{COLOURINJ} gives that the complex $\mathcal{K}(\mathcal{N}_a(\alpha))$ is isomorphic to the complex $\mathcal{K}(\mathcal{N}_{g(a)}(g[\alpha]))$ where $g[\alpha]=\{g(A):A\in\alpha\}$ and $\alpha$ is either $P(B)$ or $N(B)$. We conclude that $g(\mathrm{Sing}(B))=\mathrm{Sing}(g(B))$. Hence $\Lambda(B)=\Sigma_{a\in \mathrm{Sing}(B)}\kappa_a(B)=\Sigma_{a\in \mathrm{Sing}(B)}\kappa_{g(a)}(g(B))=\Sigma_{a\in \mathrm{Sing}(g(B))}\kappa_a(g(B))=\Lambda(g(B))$.

To prove the theorem, we consider any partition of $D$ into finitely many blocks $B_i,1\leq i\leq m$ such that $f\upharpoonright B_i$ is $pp$-definable. This is possible by an application of lemma \ref{REP}) to the set $Graph(f)$ followed projection of the finitely many blocks onto the first $n$ coordinates. Note that $D=\bigsqcup_{i=1}^m B_i \Rightarrow f(D)=\bigsqcup_{i=1}^m f(B_i)$ since $f$ is injective. Hence $\Lambda(f(D))=\Sigma_{i=1}^m\Lambda(f(B_i))=\Sigma_{i=1}^m\Lambda(B_i)=\Lambda(D)$, where the first and third equality follows by theorem \ref{t2} and the second equality follows from the previous paragraph.
\end{proof}

Now we are ready to prove a special case of the result promised at the end of section \ref{GRFOS}, which states that the Grothendieck ring of a right $\mathcal R$-module $M$ satisfying $M\equiv M^{(\aleph_0)}$ contains $\mathbb Z$ as a subgroup. This shows, in particular, that $K_0(M)$ is nontrivial in this case.
\begin{cor}\label{MAINRESULT}
Suppose $D\subseteq M^n$ is definable and $f:D\rightarrowtail D$ is a definable injection whose image is cofinite in the codomain, then $f$ is an isomorphism.
\end{cor}
\begin{proof}
We extend the function $f$ to an injective function $g:M^n\rightarrowtail M^n$ by setting $g(a)=f(a)$ if $a\in D$ and $g(a)=a$ otherwise. Now $F:=M^n\setminus Im(g)$ is finite; say it has $p$ elements. Further $\Lambda(Im(g))=\Lambda(M^n\setminus F)=\Lambda(M^n)-\Lambda(F)=-p$.

By theorem \ref{t3}, we get $\Lambda(M^n)=\Lambda(Im(g))$ since $g$ is definable injective. Hence $p=0$ and thus $g$ is an isomorphism. Since $g$ is the identity function outside $D$, we conclude that $f$ is a definable isomorphism.
\end{proof}

\subsection{Coloured global characteristics}\label{LCCC}

Let $P\in\mathcal{L}^*$ be fixed for this section. We develop the notion of localization at $P$ and local characteristic at $P$; we have developed these ideas earlier when $P$ is a singleton. After stating what we mean by a colour, we define the notion of a ``coloured global characteristic'' and outline the proof that these invariants are preserved under definable isomorphisms.

\begin{definition}
We use $\mathcal{L}_P$ to denote the meet-semilattice of all upper bounds of $P$ in $\mathcal{L}$, i.e. $\mathcal{L}_P:=\{A\in\mathcal{L}:A\supseteq P\}$. As usual, we denote the set of all finite antichains in this semilattice by $\mathcal{A}_P$.
\end{definition}

Since every element of $\mathcal{L}_P$ contains $P$, we may as well quotient out $P$ from each such element. Such a process is consistent with our earlier definition of localization since taking quotient with respect to a singleton set gives an isomorphic copy of the original set.

\begin{definitions}
We define the operator $\mathcal{Q}_P$ on the elements of $\mathcal{L}_P$ by setting $\mathcal{Q}_P(A):=p+\frac{A-p}{P-p}=\{a+(P-p):a\in A\}$ for any $p\in P$. We can clearly extend this operator to finite subsets of $\mathcal{L}_P$. Now let $\mathcal{L}_{(P)}:=\mathcal{Q}_P[\mathcal{L}_P]$. We use $\mathcal{A}_{(P)}$ to denote the set of all finite antichains in this semilattice.
\end{definitions}

It is easy to see that $\mathcal{A}_{(P)}=\mathcal{Q}_P[\mathcal{A}_P]$.

The appropriate analogue of the localization operator $\mathcal N_a:\mathcal A\rightarrow\mathcal A_a$ is a function $\mathcal{N}_P:\mathcal{A}\rightarrow\mathcal{A}_{(P)}$.
\begin{deflem}
For $\alpha\in\mathcal{A}$, we define $\mathcal{N}_P(\alpha):=\mathcal{Q}_P(\alpha\cap\mathcal{L}_P)$. As an operator on $pp$-convex sets, $\mathcal{N}_P$ preserves both unions and intersections.
\end{deflem}

The proof is easy and thus omitted.

Recall from definition \ref{SIMPCOMP} of $\mathcal K^a(\alpha)$ that the ``trivial intersections'' were precisely those which were empty or a singleton. On the other hand, ``nontrivial intersections'' were precisely those which contained the $pp$-set $\{a\}$ properly. As $\mathcal N_P$ takes values in $\mathcal A_{(P)}$, we get the correct notion of non-trivial intersections followed by the quotient operation so that the techniques developed for a singleton $P$ still remain valid. Now we are ready to state the analogue of definition \ref{SIMPCOMP}.

\begin{definition}
For $\alpha\in\mathcal{A}$, we define the \textbf{simplicial complex} of $\alpha$ \textbf{in the neighbourhood of} $P$ as the complex $\mathcal{K}(\mathcal{N}_P(\alpha))=\{\beta\subseteq\mathcal{N}_P(\alpha):|\bigcap\beta|=\infty\}$. For simplicity of notation, we denote this complex by $\mathcal{K}^P(\alpha)$.
\end{definition}

We can easily extend the notion of local characteristic at $P$ as follows.
\begin{definition}
We define the \textbf{local characteristic} of $\alpha$ at $P$ by $\kappa_P(\alpha):=\chi(\mathcal{K}^P(\alpha))-\delta(\alpha)(P)$.
\end{definition}

It can be observed that we recover the definition of the local characteristic at a point $a\in M$ by choosing $P=\{a\}$. The proofs of theorem \ref{t1} and lemma \ref{l1} go through if we replace $\kappa_a$ by $\kappa_P$. Thus we can define $\kappa_P(D)$ for arbitrary definable sets $D\subseteq M^n$.

We define the function $\kappa:\mathrm{Def}(M^n)\times\mathcal{L}^*\rightarrow\mathbb{Z}$ by setting $\kappa(D,P):=\kappa_P(D)$.

\begin{definition}
The \textbf{set of $\mathcal{L}$-singular elements} of a definable set $D\subseteq M^n$ is defined as the set $\mathrm{Sing}_\mathcal{L}(D):=\{P\in\mathcal{L}:\kappa(D,P)\neq0\}$.
\end{definition}

Fixing any partition of $D$ into blocks, it can be checked that the set $\mathrm{Sing}_\mathcal L(D)$ is contained in the nest corresponding to that partition and hence is finite. This finiteness will be used to define analogues of the global characteristic, which we call ``coloured global characteristics''.

\begin{definition}
For a given $P\in\mathcal{L}$, we define the \textbf{colour} of $P$ to be the set $\{A\in L:$ there is a bijection $f:A\cong P$ such that $Graph(f)$ is $pp$-definable $\}$. We denote the colour of $P$ by $[[P]]$.
\end{definition}

Note the significance of this definition. Theorem \ref{PPET} describes the $pp$-sets as fundamental definable sets and we are trying to classify definable sets up to definable isomorphism (definition \ref{defiso}). In fact it is sufficient to classify $pp$-sets up to $pp$-definable isomorphisms, which is the motivation behind the definition of a colour.

Let $\mathcal{X}$ denote the set of colours of elements from $\mathcal L$. We use letters $\mathfrak{A},\mathfrak{B},\mathfrak{C}$ etc. to denote the colours. It can be observed that $[[\emptyset]]$ is a singleton. We denote the colour of any singleton by $\mathfrak{U}$. We use $\mathcal{X}^*$ to denote $\mathcal{X}\setminus\{[[\emptyset]]\}$.

The global characteristic $\Lambda(D)$ is equal to $-\Sigma_{P\in\mathfrak{U}}\kappa_P(D)$ for each definable set $D$. This observation can be used to extend the notion of global characteristic.

\begin{definition}
For $\mathfrak{A}\in\mathcal{X}^*$, we define the \textbf{coloured global characteristic} with respect to $\mathfrak A$ for a definable set $D$ to be the integer $\Lambda_{\mathfrak{A}}(D):=-\Sigma_{P\in\mathfrak{A}}\kappa_P(D)$. This integer is well defined as it is equal to the finite sum $-\Sigma\{\kappa_P(D):P\in(\mathfrak{A}\cap \mathrm{Sing}_\mathcal{L}(D))\}$.
\end{definition}

The property of coloured global characteristics that we are looking for is stated in the following analogue of theorem \ref{t3}. The proof is analogous to that of \ref{t3} and thus is omitted.

\begin{theorem}\label{t4}
If $f:D\rightarrow D'$ is a definable bijection between definable sets $D,D'$, then $\Lambda_\mathfrak{A}(D)=\Lambda_\mathfrak{A}(D')$ for each $\mathfrak{A}\in\mathcal{X}^*$.
\end{theorem}

\section{Special Case: Multiplicative Structure}\label{spcasemult}
\subsection{Monoid rings}\label{MRSK}

We need the notion of an algebraic structure called a \emph{monoid ring}.
\begin{definition}
Let $(A,\star,1)$ be a commutative monoid and $S$ be a commutative ring with unity. Then we define an $L_{ring}$-structure $(S[A],0,1,+,\cdotp)$ called a \textbf{monoid ring} as follows.
\begin{itemize}
	\item $S[A]:=\{\phi:A\rightarrow S :$ the set $\mathrm{Supp}(\phi)=\{a:\phi(a)\neq 0\}$ is finite$\}$
	\item $(\phi+\psi)(a):=\phi(a)+\psi(a)$ for $a\in A$
	\item $(\phi\cdotp\psi)(a):=\Sigma_{b\star c=a}\phi(b)\psi(c)$ for $a\in A$
\end{itemize}
An element $\phi$ of $S[A]$ can be represented as a formal sum $\Sigma_{a\in A}s_a a$ where $s_a=\phi(a)$.
\end{definition}

As an example, let $A=\mathbb N$ ,equivalently the monoid $\{X^n\}_{n\geq 0}$ considered multiplicatively. Then the monoid ring $S[A]=S[\mathbb N]\cong S[X]$, the polynomial ring in one variable with coefficients from $S$.

Let the symbols $\overline{\mathcal L},\overline{\mathcal A},\overline{\mathcal X},\hdots$ etc. denote the unions $\bigcup_{n=1}^\infty\mathcal L_n$, $\bigcup_{n=1}^\infty\mathcal A_n$, $\bigcup_{n=1}^\infty\mathcal X_n,\hdots$ respectively. We shall be especially concerned with the sets $\overline{\mathcal L}^*:=\overline{\mathcal L}\setminus\{\emptyset\}$ and $\overline{\mathcal X}^*:=\overline{\mathcal X}\setminus\{[[\emptyset]]\}$.

There is a binary operation $\times:\overline{\mathcal L}^*\times\overline{\mathcal L}^*\rightarrow \overline{\mathcal L}^*$ which maps a pair $(A,B)$ to the cartesian product $A\times B$. This map commutes with the operation $[[-]]$ of taking colour i.e., whenever $[[A_1]]=[[A_2]]$ and $[[B_1]]=[[B_2]]$, we have $[[A_1\times B_1]]=[[A_2\times B_2]]$. This allows us to define a binary operation $\star:\overline{\mathcal X}^* \times \overline{\mathcal X}^* \rightarrow \overline{\mathcal X}^*$ which takes a pair of colours $(\mathfrak A,\mathfrak B)$ to $[[A\times B]]$ for any $A\in\mathfrak A,B\in\mathfrak B$. The colour $\mathfrak U$ of singletons acts as the identity element for the operation $\star$. Hence $(\overline{\mathcal X}^*,\star,\mathfrak U)$ is a monoid.

Consider the maps $\Lambda_\mathfrak A:\widetilde{\mathrm{Def}}(M)\rightarrow\mathbb Z$ for $\mathfrak A\in\overline{\mathcal X}^*$ defined by $[D]\mapsto\Lambda_\mathfrak A(D')$ for any $D'\in[D]$. These maps are well defined due to theorem \ref{t4}. We can fix some $[D]\in\widetilde{\mathrm{Def}}(M)$ and look at the set $\mathrm{Supp}([D]):=\{\mathfrak A\in\overline{\mathcal X}^*:\Lambda_\mathfrak A(D)\neq 0\}$. This set is finite since it is contained in the finite set $\{[[P]]:P\in \mathrm{Sing}_{\overline{\mathcal L}}(D)\}$. This shows that the evaluation map $ev_{[D]}:\overline{\mathcal X}^*\rightarrow\mathbb Z$ defined by $\mathfrak A\mapsto\Lambda_{\mathfrak A}([D])$ for each $[D]\in\widetilde{\mathrm{Def}}(M)$ is an element of the monoid ring $\mathbb Z[\overline{\mathcal X}^*]$.

Let us consider an example. We take $\mathcal R$ to be an infinite skew-field (i.e. a (possibly non-commutative) ring in which every nonzero element has two-sided multiplicative inverse) and $M$ to be any nonzero $\mathcal R$-vector space. This example has been studied in detail in \cite{Perera}. In this case, we have $Th(M)=Th(M)^{\aleph_0}$. Using the notion of affine dimension, it can be shown that $\overline{\mathcal X}^*\cong\mathbb N$. It has been shown that $K_0(M)\cong\mathbb Z[X]\cong\mathbb Z[\mathbb N]$. The proof in \cite{Perera} explicitly shows that the semiring $\widetilde{\mathrm{Def}}(M)$ is cancellative and is isomorphic to the semiring of polynomials in $\mathbb Z[X]$ with non-negative leading coefficients.

We will prove that a similar fact holds for an arbitrary module $M$, i.e., the structure of the Grothendieck ring $K_0(M)$ is entirely determined by the monoid $\overline{\mathcal X}^*$.
\begin{theorem}\label{FINAL}
Let $M$ be a right $\mathcal R$-module satisfying $Th(M)=Th(M)^{\aleph_0}$. Then $K_0(M)\cong\mathbb Z[\overline{\mathcal X}^*]$. In particular, $K_0(M)$ is nontrivial for every nonzero module $M$.
\end{theorem}
The proof of this theorem will occupy the next two sections.

\subsection{Multiplicative structure of $\widetilde{\mathrm{Def}}(M)$}\label{MULT}

Given $D_1\in \mathrm{Def}(M^n)$ and $D_2\in \mathrm{Def}(M^m)$, their cartesian product $D_1\times D_2\in \mathrm{Def}(M^{(n+m)})$. This shows that $\overline{\mathrm{Def}}(M)$ is closed under cartesian products. We want to show that the sets $\overline{\mathcal L}$, $\overline{\mathcal A}$, $\overline{\mathcal B}$ and $\overline{\mathcal C}$ are all closed under multiplication.

Let $P\in\mathcal L_n$ and $Q\in\mathcal L_m$. Then there are $pp$ formulas $\phi(\overline{x})$ and $\psi(\overline{y})$ defining those sets respectively. Without loss, we may assume that $\overline{x}\cap\overline{y}=\emptyset$. Now the formula $\rho(\overline{x},\overline{y})=\phi(\overline{x})\wedge\psi(\overline{y})$ is again a $pp$-formula and it defines the set $P\times Q\in\mathcal L_{n+m}$. This shows that the set $\overline{\mathcal L}$ is closed under multiplication.

Now we want to show that the product of two antichains $\alpha\in\mathcal A_n$ and $\beta\in\mathcal A_m$ is again an antichain in $\mathcal A_{n+m}$. We have natural projection maps $\pi_1:M^{n+m}\rightarrow M^n$ and $\pi_2:M^{n+m}\rightarrow M^m$ which project onto the first $n$ and the last $m$ coordinates respectively. First we observe that $(\bigcup\alpha)\times(\bigcup\beta)=\bigcup_{A\in\alpha}\bigcup_{B\in\beta}A\times B$. If either $A_1,A_2\in\alpha$ are distinct or $B_1,B_2\in\beta$ are distinct, then all the distinct elements from $\{A_i\times B_j\}_{i,j=1}^2$ are incomparable with respect to the inclusion ordering since at least one of their projections is so. Hence $\bigcup\alpha\times\bigcup\beta$ is indeed an antichain of the rank $|\alpha|\times|\beta|$. We will denote this antichain by $\alpha\times\beta$.

Given $C_1,C_2\in\overline{\mathcal C}$, we have $C_1\times C_2=\bigcup(\alpha_1\times\alpha_2)\setminus (\bigcup(\alpha_1\times\beta_2)\cup\bigcup(\beta_1\times\alpha_2))$ where $\alpha_i=P(C_i)$ and $\beta_i=N(C_i)$ for $i=1,2$. This shows that $C_1\times C_2\in\overline{\mathcal C}$ since $\overline{\mathcal A}$ is closed under both products and unions. Furthermore, we observe that $P(C_1\times C_2)=P(C_1)\times P(C_2)$. This in particular shows that the set $\overline{\mathcal B}$ of blocks is also closed under products.

\begin{lemma}\label{localcharmult}
Let $P,Q\in\overline{\mathcal L}$ and $\alpha,\beta\in\overline{\mathcal A}$. Then $\kappa_{P\times Q}(\alpha\times \beta)=-\kappa_P(\alpha)\kappa_Q(\beta)$.
\end{lemma}
\begin{proof}
First assume that $\delta(\alpha)(P)=\delta(\beta)(Q)=1$. Then observe that
\begin{equation}\label{discomp}\mathcal K^{P\times Q}(\alpha\times\beta)\cong\mathcal K^P(\alpha)\boxtimes\mathcal K^Q(\beta).\end{equation}
Hence we have
\begin{eqnarray*}
\kappa_{P\times Q}(\alpha\times\beta)&=&\chi(\mathcal K^{P\times Q}(\alpha\times\beta))-1\\
    &=&\chi(\mathcal K^P(\alpha))+\chi(\mathcal K^Q(\beta))-\chi(\mathcal K^P(\alpha))\chi(\mathcal K^Q(\beta))-1\\
    &=&(\kappa_P(\alpha)+1)+(\kappa_Q(\beta)+1)-(\kappa_P(\alpha)+1)(\kappa_Q(\beta)+1)-1\\
    &=&-\kappa_P(\alpha)\kappa_Q(\beta)
\end{eqnarray*}
The first and third equality is by definition of the local characteristic and the second is by equation (\ref{ecdp}) of lemma \ref{ecdpl} applied to (\ref{discomp}).

In the remaining case when either $\delta(\alpha)(P)$ or $\delta(\beta)(Q)$ is $0$, we have $\delta(\alpha\times\beta)(P\times Q)=0$. Hence $\kappa_{P\times Q}(\alpha\times\beta)=0$ and  either $\kappa_P(\alpha)$ or $\kappa_Q(\beta)$ is $0$. This gives the necessary identity and thus completes the proof in all cases.
\end{proof}

The aim of this section is to prove the following theorem.
\begin{theorem}\label{t5}
The map $ev:\widetilde{\mathrm{Def}}(M)\rightarrow\mathbb Z[\overline{\mathcal X}^*]$ defined by $[D]\mapsto ev_{[D]}$ is a semiring homomorphism.
\end{theorem}
\begin{proof} We have already seen that $ev$ is additive, since each $\Lambda_{\mathfrak A}$ is. So it remains to show that it is multiplicative.

We have observed that the set $[\overline{\mathcal A}]$ is a monoid with respect to cartesian product, the isomorphism class of a singleton being the identity for the multiplication. So we will first show that $ev:[\overline{\mathcal A}]\rightarrow\mathbb Z[\overline{\mathcal X}^*]$ is a multiplicative monoid homomorphism.

Let $\alpha,\beta\in\overline{\mathcal A}$ be fixed. Note that

\begin{equation}\label{singincl}S:=\mathrm{Sing}_{\overline{\mathcal L}}(\alpha\times\beta)\subseteq\{P\times Q:P\in \mathrm{Sing}_{\overline{\mathcal L}}(\alpha),Q\in \mathrm{Sing}_{\overline{\mathcal L}}(\beta)\}.\end{equation}

We need to show that $ev_{[\alpha]}\cdotp ev_{[\beta]}=ev_{[\alpha\times\beta]}$ as maps on $\overline{\mathcal X}^*$. This is equivalent to  $ev_{[\alpha\times\beta]}(\mathfrak C)=\sum_{\mathfrak A\star\mathfrak B=\mathfrak C}ev_{[\alpha]}(\mathfrak A) ev_{[\beta]}(\mathfrak B)$ for each $\mathfrak C\in\overline{\mathcal X}^*$. Using the definition of the evaluation map, it is enough to check that $\Lambda_{\mathfrak C}([\alpha\times\beta])=\sum_{\mathfrak A\star\mathfrak B=\mathfrak C}\Lambda_{\mathfrak A}([\alpha]) \Lambda_{\mathfrak B}([\beta])$ for each $\mathfrak C\in\overline{\mathcal X}^*$.

The left hand side of the above equation is
\begin{eqnarray*}
  \Lambda_{\mathfrak C}([\alpha\times\beta]) &=& -\sum_{R\in\mathfrak C}\kappa_R(\alpha\times\beta)\\
  &=& -\sum_{R\in(\mathfrak C\cap S)}\kappa_R(\alpha\times\beta)\\
  &=& \sum_{R\in(\mathfrak C\cap S)}\kappa_{\pi_1(R)}(\alpha)\kappa_{\pi_2(R)}(\beta)
\end{eqnarray*}
The last equality is given by the lemma \ref{localcharmult} since, by (\ref{singincl}), every $R\in\mathfrak C\cap S$ can be written as $R=\pi_1(R)\times\pi_2(R)$. The right hand side is
\begin{eqnarray*}
  \sum_{\mathfrak A\star\mathfrak B=\mathfrak C}\Lambda_{\mathfrak A}([\alpha])\Lambda_{\mathfrak B}([\beta]) &=& \sum_{\mathfrak A\star\mathfrak B=\mathfrak C}\left(-\sum_{P\in\mathfrak A}\kappa_P(\alpha)\right)\left(-\sum_{Q\in\mathfrak B}\kappa_Q(\beta)\right) \\
  \ &=& \sum_{\mathfrak A\star\mathfrak B=\mathfrak C}\sum_{P\in\mathfrak A,Q\in\mathfrak B}\kappa_P(\alpha)\kappa_Q(\beta)
\end{eqnarray*}

Using the definition of $\mathrm{Sing}_{\overline{\mathcal L}}(-)$, we observe that the final expressions on both sides are equal. This completes the proof that $ev$ is a multiplicative monoid homomorphism on $[\overline{\mathcal A}]$.

Now we will show that $ev$ is also multiplicative on the monoid $[\overline{\mathcal C}]$. Let $C_1,C_2$ be cells with $\alpha_i=P(C_i)$ and $\beta_i=N(C_i)$ for each $i=1,2$. Then $C_1\times C_2=\bigcup(\alpha_1\times\alpha_2)\setminus (\bigcup(\alpha_1\times\beta_2)\cup\bigcup(\beta_1\times\alpha_2))$. We also know that $ev_{[C]}=ev_{P(C)}-ev_{N(C)}$ for each cell $C$.

We need to show that $\Lambda_{\mathfrak C}(C_1\times C_2)=\sum_{\mathfrak A\star\mathfrak B=\mathfrak C}\Lambda_{\mathfrak A}([C_1])\Lambda_{\mathfrak B}([C_2])$ for each $\mathfrak C\in\overline{\mathcal X}^*$. Now we have
\begin{equation*}
    \Lambda_{\mathfrak C}(C_1\times C_2)=\Lambda_{\mathfrak C}(\alpha_1\times\alpha_2)-\Lambda_{\mathfrak C}((\alpha_1\times\beta_2)\vee(\beta_1\times\alpha_2))
\end{equation*}
and we also have
\begin{eqnarray*}
   \sum_{\mathfrak A\star\mathfrak B=\mathfrak C}\Lambda_{\mathfrak A}([C_1])\Lambda_{\mathfrak B}([C_2]) &=& \sum_{\mathfrak A\star\mathfrak B=\mathfrak C}(\Lambda_{\mathfrak A}([\alpha_1])-\Lambda_{\mathfrak A}([\beta_1]))(\Lambda_{\mathfrak B}([\alpha_2])-\Lambda_{\mathfrak B}([\beta_2])) \\
   &=& \Lambda_{\mathfrak C}(\alpha_1\times\alpha_2)+\Lambda_{\mathfrak C}(\beta_1\times\beta_2)-\Lambda_{\mathfrak C}(\beta_1\times\alpha_2)-\Lambda_{\mathfrak C}(\alpha_1\times\beta_2)
\end{eqnarray*}
Therefore we need to show
\begin{equation*}
   \Lambda_{\mathfrak C}((\alpha_1\times\beta_2)\vee(\beta_1\times\alpha_2))+\Lambda_{\mathfrak C}(\beta_1\times\beta_2)=\Lambda_{\mathfrak C}(\alpha_1\times\beta_2)+\Lambda_{\mathfrak C}(\beta_1\times\alpha_2).
\end{equation*}
This is true by theorem \ref{t1} since we have $(\alpha_1\times\beta_2)\wedge(\beta_1\times\alpha_2)=(\beta_1\times\beta_2)$.

In the last step, we show that $ev_{[D_1\times D_2]}=ev_{[D_1]}\cdotp ev_{[D_2]}$ for arbitrary definable sets $D_1,D_2$. Let $[D_1]=\sum_{i=1}^k[B_{1i}]$ and $[D_2]=\sum_{j=1}^l[B_{2j}]$ be obtained from any decompositions of $D_1$ and $D_2$ into blocks. Then $[D_1\times D_2]=\sum_{i=1}^{k}\sum_{j=1}^{l}[B_{1i}\times B_{2j}])$. For each $\mathfrak C\in\overline{\mathcal X}^*$, we have
\begin{eqnarray*}
  ev_{[D_1]}\cdotp ev_{[D_2]}(\mathfrak C) &=& \sum_{\mathfrak A\star\mathfrak B=\mathfrak C}\Lambda_{\mathfrak A}([D_1])\Lambda_{\mathfrak B}([D_2]) \\
   &=& \sum_{\mathfrak A\star\mathfrak B=\mathfrak C}\left(\sum_{i=1}^{k}\Lambda_{\mathfrak A}([B_{1i}])\right)\left(\sum_{j=1}^{l}\Lambda_{\mathfrak B}([B_{2j}])\right) \\
   &=& \sum_{i=1}^{k}\sum_{j=1}^{l}\sum_{\mathfrak A\star\mathfrak B=\mathfrak C}\Lambda_{\mathfrak A}([B_{1i}])\Lambda_{\mathfrak B}([B_{2j}]) \\
   &=& \sum_{i=1}^{k}\sum_{j=1}^{l}\Lambda_{\mathfrak C}([B_{1i}\times B_{2j}]) \\
   &=& \Lambda_{\mathfrak C}(\sum_{i=1}^{k}\sum_{j=1}^{l}[B_{1i}\times B_{2j}]) \\
   &=& ev_{[D_1\times D_2]}(\mathfrak C).
\end{eqnarray*}
This completes the proof showing $ev$ is a semiring homomorphism.
\end{proof}

\subsection{Computation of the Grothendieck ring}\label{COMPUTATION}

In the previous section, we showed that $ev:\widetilde{\mathrm{Def}}(M)\rightarrow\mathbb Z[\overline{\mathcal X}^*]$ is a semiring homomorphism. Since the codomain of this map is a ring, it factorizes through the unique homomorphism of cancellative semirings $\widetilde{ev}:\widetilde{\widetilde{\mathrm{Def}}(M)}\rightarrow\mathbb Z[\overline{\mathcal X}^*]$ where $\widetilde{\widetilde{\mathrm{Def}}(M)}$ is the quotient semiring of $\widetilde{\mathrm{Def}}(M)$ obtained as in theorem \ref{QUOCONST}. Our next aim is to prove the following lemma.

\begin{lemma}\label{INJEV}
The map $\widetilde{ev}:\widetilde{\widetilde{\mathrm{Def}}(M)}\rightarrow\mathbb Z[\overline{\mathcal X}^*]$ is injective.
\end{lemma}
\begin{proof}
We will prove this lemma in several steps. First we will identify a subset of $\overline{\mathrm{Def}}(M)$ where the restriction of the evaluation function is injective.

Let $\mathcal U=\{\alpha\in\overline{\mathcal A}:A_1\cap A_2=\emptyset$ for all distinct $A_1,A_2\in\alpha\}$. Then it can be easily checked that $\Lambda_{\mathfrak A}(\alpha)=|\alpha\cap\mathfrak A|$ for each $\mathfrak A\in\overline{\mathcal X}^*$ and $\alpha\in\mathcal U$. Hence if $ev_{[\alpha]}=ev_{[\beta]}$ for any $\alpha,\beta\in\mathcal U$, then we have $[\alpha]=[\beta]$. This proves that the map $ev$ is itself injective on $\mathcal U$.

Given any $[D_1],[D_2]\in\widetilde{\mathrm{Def}}(M)$ such that $ev_{[D_1]}=ev_{[D_2]}$, we will find some $[X]\in\widetilde{\mathrm{Def}}(M)$ such that $[D_1]+[X]=[\alpha']$ and $[D_2]+[X]=[\beta']$ for some $\alpha',\beta'\in\mathcal U$. Then we get $ev_{[\alpha']}=ev_{[D_1]}+ev_{[X]}=ev_{[D_2]}+ev_{[X]}=ev_{[\beta']}$ and hence we will be done by the previous paragraph.

\textbf{Claim:} It is sufficient to assume $[D_1],[D_2]\in[\overline{\mathcal A}]$.

Let $[D_1]=\sum_{i=1}^k[B_{1i}]$ and $[D_2]=\sum_{j=1}^l[B_{2j}]$ be obtained from any decompositions of $D_1$ and $D_2$ into blocks. We have $[P(B)]=[B]+[N(B)]$ for any $B\in\overline{\mathcal B}$. Therefore if we choose $[Y]=\sum_{i=1}^k[N(B_{1i})]+\sum_{j=1}^l[N(B_{2j})]$, we get $[D_1]+[Y]=\sum_{i=1}^l[P(B_{1i})]+ \sum_{j=1}^l[N(B_{2j})]$ and $[D_2]+[Y]=\sum_{i=1}^l[N(B_{1i})]+\sum_{j=1}^l[P(B_{2j})]$. Hence both $[D_1]+[Y],[D_2]+[Y]\in[\overline{\mathcal A}]$. This finishes the proof of the claim.

Now let $\alpha,\beta\in\overline{\mathcal A}$ be such that $ev_{[\alpha]}=ev_{[\beta]}$. We describe an algorithm which terminates in finitely many steps and yields some $[X]$ such that $[\alpha]+[X],[\beta]+[X]\in[\mathcal U]$. Before stating the algorithm, we define a \textbf{complexity function} $\Gamma:\overline{\mathcal A}\rightarrow\mathbb N$. For each antichain $\alpha$, the complexity $\Gamma(\alpha)$ is defined to be the maximum of the lengths of chains in the smallest nest corresponding to $\alpha$, where the length of a chain is the number of elements in it. Note that $\Gamma(\alpha)\leq 1$ if and only if $\alpha\in\mathcal U$.

Let $\alpha=\{A_1,A_2,\hdots,A_k\}$ be any enumeration and let $\alpha_i=\{A_1,A_2,\hdots,A_i\}$ for each $1\leq i\leq k$. Similarly choosing an enumeration $\beta=\{B_1,B_2,\hdots,B_l\}$, we define $\beta_j$ for each $1\leq j\leq l$. Then we observe that $\bigcup\alpha=\bigsqcup_{i=1}^k \mathrm{Core}_{\alpha_i}(A_i)$ and $\bigcup\beta=\bigsqcup_{j=1}^l \mathrm{Core}_{\beta_j}(B_j)$. Now each $\mathrm{Core}_{\alpha_i}(A_i)$ is a block, which can be completed to a $pp$-set if we take its (disjoint) union with $N(\mathrm{Core}_{\alpha_i}(A_i))$. This can be written as the equation $[A_i]=[\mathrm{Core}_{\alpha_i}(A_i)]+ [N(\mathrm{Core}_{\alpha_i}(A_i))]$. If $\bigcup\alpha\subseteq M^n$, we consider $M^{nk}$ and inject $\mathrm{Core}_{\alpha_i}(A_i)$ in the obvious way into the $i^{th}$ copy of $M^n$ in $M^{nk}$ for each $i$. This gives us a definable set definably isomorphic to $\bigcup\alpha$. The advantage of this decomposition is that we can also add an isomorphic copy of $N(\mathrm{Core}_{\alpha_i}(A_i))$ at the appropriate place for each $i$ and obtain a new antichain representing $\sum_{i=1}^k [A_i]$.

Repeating the same procedure for $\beta$ yields a representative of $\sum_{j=1}^l [B_j]$. In order to maintain the evaluation function on both sides, we add disjoint copies of the antichains $N(\mathrm{Core}_{\alpha_i}(A_i))$, $N(\mathrm{Core}_{\beta_j}(B_j))$ to both sides. So we choose $[W]=\sum_{i=1}^k[N(\mathrm{Core}_{\alpha_i}(A_i))]+ \sum_{j=1}^l [N(\mathrm{Core}_{\beta_j}(B_j))]$, hence $[\alpha]+[W],[\beta]+[W]$ are both in $[\overline{\mathcal A}]$ so that the particular antichains $\alpha',\beta'$ in these classes we constructed above satisfy $\Gamma((\bigcup\alpha')\sqcup(\bigcup\beta'))<\Gamma((\bigcup\alpha)\sqcup(\bigcup\beta))$. The inequality holds since we isolate the maximal elements of the nest corresponding to $(\bigcup\alpha)\cup(\bigcup\beta)$ in the process.

We repeat this process, inducting on the complexity of the antichains, till the disjoint union of the pair of antichains in the output lies in $\mathcal U$. Since the complexity decreases at each step, this algorithm terminates in finitely many steps. The required $[X]$ is the sum of the $[W]$'s obtained at each step. This finishes the proof of the injectivity of the map $\widetilde{ev}$.
\end{proof}

Finally we are ready to prove theorem \ref{FINAL} regarding the structure of the Grothendieck ring $K_0(M)$.

\begin{proof} (of Theorem \ref{FINAL}) It is easy to observe that the image of $\mathcal U$ under the evaluation map is the monoid semiring $\mathbb N[\overline{\mathcal X}^*]$. The Grothendieck ring $K_0(\mathbb N[\overline{\mathcal X}^*])$ is clearly isomorphic to the monoid ring $\mathbb Z[\overline{\mathcal X}^*]$.

Since the map $\widetilde{ev}$ is injective by lemma \ref{INJEV} and $\mathbb N[\overline{\mathcal X}^*]\subseteq Im(\widetilde{ev})\subseteq\mathbb Z[\overline{\mathcal X}^*]$, we have $K_0(M)=K_0(Im(\widetilde{ev}))\cong\mathbb Z[\overline{\mathcal X}^*]$ by the universal property of $K_0$ in theorem \ref{GRCONSTR}.
\end{proof}

\section{General Case}\label{gencase}
\subsection{Finite indices of $pp$-pairs}\label{TNTA}

So far we have considered the Grothendieck ring of a right $\mathcal R$-module $M$ whose theory $T:=Th(M)$ satisfies $T=T^{\aleph_0}$. From this section onwards we remove this condition and work with an arbitrary right $\mathcal R$-module $M$.

We continue to use the notations $\mathcal L_n,\mathcal P_n,\mathcal A_n,\mathcal X_n$ to denote the set of all $pp$-subsets of $M^n$, the set of all finite subsets of $\mathcal L_n$, the set of all finite antichains in $\mathcal L_n$ and the set of all $pp$-isomorphism classes (colours) in $\mathcal L_n$ respectively. We still use the representation theorem \ref{REP}, but lemma \ref{NLU} is unavailable to obtain the uniqueness - proposition \ref{UNIQUE1}. As a result we do not have a bijection between the set of all $pp$-convex sets, which we denote by $\mathcal O_n$, and the set $\mathcal A_n$. The elements of the set $\mathcal C_n:=\{(\bigcup\alpha)\setminus(\bigcup\beta)| \alpha,\beta\in\mathcal A_n,\ \bigcup\beta\subsetneq\bigcup\alpha\}$ will be called cells. The cells allowing a representation of the form $P\setminus\bigcup\beta$ for some $P\in\mathcal L_n$ and $\beta\in\mathcal A$ such that $P\subsetneq\bigcup\beta$ will be called blocks and the set of all blocks in $\mathcal C_n$ is denoted by $\mathcal B_n$.

Let $(-)^\circ:\mathcal L_n\rightarrow \mathcal L_n$ denote the function which takes a coset $P$ to the subgroup $P^\circ:=P-p$, where $p\in P$ is any element. We use $\mathcal L_n^\circ$ to denote the image of this function, i.e. the set of all $pp$-definable subgroups. Let $\sim_n$ denote a relation on $\mathcal L_n^\circ$ defined by $P\sim_nQ$ if and only if $[P:P\cap Q]+[Q:P\cap Q]<\infty$. This is the \textbf{commensurability relation} and it can be easily checked to be an equivalence relation. We can extend this relation to all elements of $\mathcal L_n$ using the same definition if we set the index $[P:Q]:=[P^\circ:P^\circ\cap Q^\circ]$ for all $P,Q\in\mathcal L_n$. Let $\mathcal Y_n$ denote the set of all commensurability equivalence classes of $\mathcal L_n$ (\textbf{bands} for short). We use capital bold letters $\mathbf{P},\mathbf{Q},\cdots$ etc. to denote bands. The equivalence class (band) of $P$ will be denoted by the corresponding bold letter $\mathbf{P}$.

Now we fix some $n\geq 1$ and drop all the subscripts as usual. Note that, in the special case, a band is just the collection of all cosets of a $pp$-subgroup. In particular any two distinct elements of a band are disjoint. This `discreteness' has been exploited heavily in all the proofs for the special case. We need to work hard to set up the technical machinery for defining the local characteristics. The proofs for the general case will be similar to those for the special case once we obtain the required discreteness condition.

Let $\mathbf{P}\in\mathcal Y$. It can be easily checked that if $P,Q\in\mathbf{P}$ and $P\cap Q\neq\emptyset$ then $P\cap Q\in\mathbf{P}$ i.e. $\mathbf{P}$ is closed under intersections which are nonempty. By definition of the index, it is also clear that if $P\in\mathbf{P}$ and $a\in M^n$, then $a+P\in\mathbf{P}$. Let $\mathcal A(\mathbf{P}),\mathcal P(\mathbf{P})$ and $\mathcal O(\mathbf{P})$ denote the sets of all finite antichains in $\mathbf{P}$, finite subsets of $\mathbf{P}$ and unions of finite subsets of $\mathbf{P}$ respectively.

We have the following analogue of proposition \ref{UNIQUE1} for $pp$-convex sets. The proof is omitted as it is similar to the $\mathrm{T=T^{\aleph_0}}$ case.

\begin{pro}
Let $X\in\mathcal O$. Then the set $S(X):=\{\mathbf{P}\in\mathcal Y: \exists \alpha\in\mathcal A\ (P\in~\alpha,$ $ \bigcup\alpha=X)\}$ is finite. Furthermore for any two $\alpha,\beta\in\mathcal A$ such that $\bigcup\alpha=\bigcup\beta=X$ and each $\mathbf{P}\in S(X)$, we have $\bigcup(\alpha\cap\mathbf{P})=\bigcup(\beta\cap\mathbf{P})$. Thus $X$ is uniquely determined by the family $\{X_\mathbf{P}:=\bigcup(\alpha\cap\mathbf{P}) \in\mathcal O(\mathbf{P})\mid\mathbf{P}\in S(X)\}$ for any $\alpha\in\mathcal A$ such that $\bigcup\alpha=X$.
\end{pro}

Given some $X\in\mathcal O(\mathbf{P})$ there could be two different $\alpha,\beta\in\mathcal A(\mathbf{P})$ such that $\bigcup\alpha=\bigcup\beta=X$. The nests corresponding to such antichains could have entirely different (semilattice) structures. The following proposition gives us a way to obtain an antichain $\alpha$ representing $X$ such that if $A,B\in\alpha$ and $A\neq B$, then $A\cap B=\emptyset$.

\begin{pro}
Let $X\in\mathcal O(\mathbf{P})$. Then for any $\alpha\in\mathcal A(\mathbf{P})$ such that $\bigcup\alpha=X$, there is some $\mathbf{P}(\alpha)\in\mathbf{P}^\circ$ such that $X$ is a finite union of distinct cosets of $\mathbf{P}(\alpha)$.
\end{pro}
\begin{proof}
Choose $\mathbf{P}(\alpha)=\bigcap\{Q^\circ:Q\in\alpha\}$ and observe that $\mathbf{P}(\alpha)\in\mathbf{P}$ since $\mathbf{P}$ is closed under finite nonempty intersections.
\end{proof}

The previous two propositions together imply that we can always find a `nice' antichain representing the given $pp$-convex set. The following definition describes what we mean by this.
\begin{definition}
A finite set $\alpha\in\mathcal P$ is said to be in \textbf{discrete form} if $\alpha\cap\mathbf{P}$ consists of finitely many cosets of a fixed element of $\mathbf{P}^\circ$, denoted $\mathbf{P}(\alpha)$, for each $\mathbf{P}\in\mathcal Y$. The set of all finite sets $\alpha\in\mathcal P$ in discrete form will be denoted by $\mathcal P^d$ and the set of all antichains in discrete form will be denoted by $\mathcal A^d$.
\end{definition}

We would like to define the local characteristics for the elements of $\mathcal P^d$ as before and show that they satisfy the conclusion of theorem \ref{t1}. We will restrict our attention only to those $\alpha\in\mathcal P^d$ such that $\alpha=\hat{\alpha}$ (i.e. the nest corresponding to $\alpha$ is $\alpha$ itself). We denote the set of all such finite sets by $\hat{\mathcal P}^d$. Since we will deal with finite index subgroup pairs in $\mathcal L^\circ$, we will need more conditions on compatibility of $P$ and $\alpha$ as stated in the following definition.

\begin{definition}
A finite family $\mathcal F$ of elements of $\mathcal P$ is called \textbf{compatible} if $\mathcal F\subseteq\hat{\mathcal P}^d$ and for all $\alpha,\beta\in\mathcal F$ and $\mathbf{P}\in\mathcal Y$, we have $\mathbf{P}(\alpha)=\mathbf{P}(\beta)$ whenever $\mathbf{P}\cap\alpha,\mathbf{P}\cap\beta\neq\emptyset$. Furthermore, we say that $P\in\mathcal L$ is \textbf{compatible with} a finite family $\mathcal F$ of elements of $\mathcal P$ if $\mathcal F$ is compatible and $P\in\bigcup\mathcal F$.
\end{definition}

It is very easy to observe that given any finite family $\{X_1,X_2,\hdots,X_k\}$ of $pp$-convex sets, we can obtain a compatible family $\{\alpha_1,\alpha_2,\hdots,\alpha_k\}$ of antichains such that $\bigcup\alpha_i=X_i$ for each $i$. Finally we are ready to define the local characteristics in this set-up.

\begin{definition}
Let $P\in\mathcal L$ be compatible with a family $\mathcal F$ and let $\alpha\in\mathcal F$. We associate an abstract simplicial complex $\mathcal K^P(\alpha)$ with the pair $(\alpha,P)$ by setting $\mathcal K^P(\alpha):=\{\beta\subseteq\alpha: \beta\neq\emptyset,\,\bigcap\beta\supsetneq P\}$. We define the \textbf{local characteristic} $\kappa_P$ by the formula $\kappa_P(\alpha):=\chi(\mathcal K^P(\alpha))-\delta(\alpha)(P)$.
\end{definition}

Now we are ready to state the analogue of theorem \ref{t1} and it has essentially the same proof. The previous statement is justified because we have carefully developed the idea of a compatible family to avoid finite index pairs of $pp$-subgroups. Since we achieve discreteness simultaneously for any finite family of antichains, no changes in the proof of theorem \ref{t1} are necessary.

\begin{theorem}\label{t1general}
Let $X,Y\in\mathcal O$. Then $X\cup Y,X\cap Y\in\mathcal O$. For any compatible family $\mathcal F:=\{\alpha_1,\alpha_2,\beta_1,\beta_2\}$ such that $\bigcup\alpha_1=X$, $\bigcup\alpha_2=Y$, $\bigcup\beta_1=X\cup Y$ and $\bigcup\beta_2=X\cap Y$ and any $P\in\mathcal L$ compatible with $\mathcal F$, we have
\begin{equation*}
    \kappa_P(\alpha_1)+\kappa_P(\alpha_2)=\kappa_P(\beta_1)+\kappa_P(\beta_2).
\end{equation*}
\end{theorem}

We observe that the set $\overline{\mathcal A^d}$ is closed under cartesian products and thus we have the following analogue of lemma \ref{localcharmult} with the same proof.
\begin{lemma}\label{localcharmultgeneral}
Let $P,Q\in\overline{\mathcal L}$ be compatible with $\{\alpha,\beta\}\subseteq\overline{\mathcal A^d}$. Then
\begin{equation*}
\kappa_{P\times Q}(\alpha\times \beta)=-\kappa_P(\alpha)\kappa_Q(\beta).
\end{equation*}
\end{lemma}

\subsection{The invariants ideal}\label{II}

Once again, we use the notations $\overline{\mathcal L},\overline{\mathcal X}$ to denote the unions $\bigcup_{n=1}^{\infty}\mathcal L_n,$ $\bigcup_{n=1}^{\infty}\mathcal X_n$ etc. and set $\overline{\mathcal L}^*=\overline{\mathcal L}\setminus\{\emptyset\}, \overline{\mathcal X}^*=\overline{\mathcal X}\setminus\{[[\emptyset]]\}$ where $[[-]]:\overline{\mathcal L}\rightarrow\overline{\mathcal X}$ is the map taking a $pp$-set to its colour. Now, $\overline{\mathcal X}^*$ is a multiplicative monoid and we consider the monoid ring $\mathbb Z[\overline{\mathcal X}^*]$.

In the case when $\mathrm{T\neq T^{\aleph_0}}$, there are $P,Q\in\mathcal L_n$ such that $1<\mathrm{Inv}(M;P,Q)<\infty$ for each $n\geq 1$. We can assume without loss that $0\in Q\subseteq P$. Now we define an ideal of the monoid ring, called \textbf{the invariants ideal}, which encodes these invariants. The following proposition is the motivation.

\begin{pro}\label{partitionfurther}
Let $\mathbf{P}\in\mathcal Y_n$ and $X\in\mathcal O(\mathbf{P})$. For any $\alpha,\beta\in\mathcal A^d_n$ with $\bigcup\alpha=\bigcup\beta=X$, we have
\begin{equation*}
    [\mathbf{P}(\alpha):\mathbf{P}(\beta)]|\alpha\cap\mathbf{P}| =[\mathbf{P}(\beta):\mathbf{P}(\alpha)]|\beta\cap\mathbf{P}|
\end{equation*}
\end{pro}
\begin{proof}
Partition those cosets of both $\mathbf{P}(\alpha)$ and of $\mathbf{P}(\beta)$ which are contained in $X$ into cosets of $\mathbf{P}(\alpha)\cap\mathbf{P}(\beta)$ to get the required equality.
\end{proof}

\begin{definition}
Let $\delta_{\mathfrak A}:\overline{\mathcal X}^*\rightarrow\mathbb Z$ denote the characteristic function of the colour $\mathfrak A$ for each $\mathfrak A\in\overline{\mathcal X}^*$. We define \textbf{the invariants ideal $\mathcal J$} of the monoid ring $\mathbb Z[\overline{\mathcal X}^*]$ to be the ideal generated by the set
\begin{equation*}
    \{\delta_{[[P]]}=[P:Q]\delta_{[[Q]]}: P,Q\in\overline{\mathcal L},\ P\supseteq Q,\ \mathrm{Inv}(M;P,Q)<\infty\}.
\end{equation*}
\end{definition}

The main aim of this section is to prove the following theorem.
\begin{theorem}\label{FINALgeneral}
For every right $\mathcal R$-module $M$, we have
\begin{center}
$K_0(M)\cong\mathbb Z[\overline{\mathcal X}^*]/\mathcal J$.
\end{center}
\end{theorem}

We have proved this theorem when $\mathrm{T=T^{\aleph_0}}$ since the invariants ideal is trivial in that case.

Let $\overline{\mathcal Y}=\bigcup_{n=1}^\infty\mathcal Y_n$. Given $\mathfrak A\in\overline{\mathcal X}^*$, we define $\mathcal Y(\mathfrak A):=\{\mathbf{P}\in\overline{\mathcal Y}:\mathbf{P}\cap\mathfrak A\neq\emptyset\}$. In order to define the global characteristics in this case, we need to find the set over which they vary. Let $\mathfrak A,\mathfrak B\in\overline{\mathcal X}^*$. We say that $\mathfrak A\approx\mathfrak B$ if and only if $\mathcal Y(\mathfrak A)\cap\mathcal Y(\mathfrak B)\neq\emptyset$. This relation is reflexive and symmetric. We use $\approx$ again to denote its transitive closure. The $\approx$-equivalence class of $\mathfrak A$ will be denoted by $\widetilde{\mathfrak A}$.

\begin{definition}
Let $\mathfrak A\in\overline{\mathcal X}^*$. Define the \textbf{colour class group} $\mathcal R(\widetilde{\mathfrak A})$ as the quotient of the free abelian group $\mathbb Z\langle \delta_{\mathfrak A}:\mathfrak A\in\widetilde{\mathfrak A}\rangle$ by the subgroup $\mathcal J(\widetilde{\mathfrak A})$ generated by the relations $\{\delta_{[[P]]}=[P:Q]\delta_{[[Q]]}: P,Q\in\bigcup\widetilde{\mathfrak A},\ P\supseteq Q\}$.
\end{definition}

It can be observed that the underlying abelian group of the monoid ring $\mathbb Z[\overline{\mathcal X}^*]$ is formed by taking the quotient of the direct sum of the free abelian groups $\mathbb Z\langle \delta_{\mathfrak A}:\mathfrak A\in\widetilde{\mathfrak A}\rangle$, one for each equivalence class of colours, by the multiplicative relations of the monoid $\overline{\mathcal X}^*$. Furthermore, the set $\bigcup\{\mathcal J(\widetilde{\mathfrak A}):\widetilde{\mathfrak A}\in\overline{\mathcal X}^*\}$ generates the ideal $\mathcal J$ in this ring.

The discussion in the previous paragraph suggests to us to isolate the information in the evaluation map into different global characteristics, one for each colour class. These maps take values in the corresponding colour class group. We define the \textbf{global characteristic} $\Lambda_{\widetilde{\mathfrak A}}$ corresponding to $\widetilde{\mathfrak A}$ as the function $\overline{\hat{\mathcal P}^d}\rightarrow\mathcal R(\widetilde{\mathfrak A})$ given by $\alpha\mapsto-\sum_{\mathfrak A\in\widetilde{\mathfrak A}}\left(\sum_{P\in\mathfrak A}\kappa_P(\alpha)\right)\delta_{\mathfrak A}$.

The following result is an easy corollary of proposition \ref{partitionfurther}. It states that the global characteristics depend only on the $pp$-convex sets and not on their representations as antichains.
\begin{cor}\label{glogeneralconvex}
Let $X\in\overline{\mathcal O}$ and $\alpha,\beta\in\overline{\hat{\mathcal P}^d}$ be such that $\bigcup\alpha=\bigcup\beta=X$. Then $\Lambda_{\widetilde{\mathfrak A}}(\alpha)=\Lambda_{\widetilde{\mathfrak A}}(\beta)$ for each $\mathfrak A\in\overline{\mathcal X}^*$.
\end{cor}

This finishes the technical setup for the general case when the theory $T$ of the module $M$ does not necessarily satisfy $T=T^{\aleph_0}$. The antichains in discrete form behave as if the theory satisfies $T=T^{\aleph_0}$, the bands allow us to go down (via intersections) so that any finite family can be converted to a compatible family and the notion of compatibility allows us to do appropriate local analysis. The local data can be pasted together using the information coded in the colour class groups.

Now we give some important definitions and state results from the special case $\mathrm{T=T^{\aleph_0}}$ in a form compatible with the general case. The proofs of these results are omitted since they are similar to their special counterparts; the basic ingredients are provided by lemma \ref{NLU}, theorem \ref{t1general}, lemma \ref{localcharmultgeneral} and corollary \ref{glogeneralconvex}. The necessary change is to deal only with antichains which are in discrete form.

Since cells are the difference sets of two $pp$-convex sets, we can obtain a compatible family $\{\alpha,\beta\}$ for any $C\in\overline{C}$ such that $C=\bigcup\alpha\setminus\bigcup\beta$.

\begin{definition}
Let $C\in\overline{\mathcal C}$ and $\mathfrak A$ be a colour. We define the global characteristic $\Lambda_{\widetilde{\mathfrak A}}(C):=\Lambda_{\widetilde{\mathfrak A}}(\alpha)-\Lambda_{\widetilde{\mathfrak A}}(\beta)\in\mathcal R(\widetilde{\mathfrak A})$ for any compatible family $\{\alpha,\beta\}$ representing $C$.
\end{definition}

The following theorem is the analogue of theorem \ref{t2} and uses the inductive version of \ref{t1general} in its proof.
\begin{theorem}\label{t2general}
If $\{B_i:1\leq i\leq l\}, \{B_j':1\leq j\leq m\}$ are two finite families of pairwise disjoint blocks such that $\bigsqcup_{i=1}^l B_i=\bigsqcup_{j=1}^m B_j'$, then $\Sigma_{i=1}^l\Lambda_{\widetilde{\mathfrak A}} (B_i)=\Sigma_{j=1}^m\Lambda_{\widetilde{\mathfrak A}}(B_j')$ for every $\mathfrak A\in\overline{\mathcal X}^*$.
\end{theorem}

This theorem allows us to extend the definition of global characteristics to all sets in $\overline{\mathrm{Def}}(M)$. Moreover the following theorem, the proof of which is an easy adaptation of that of theorem \ref{t3}, states that each of them is preserved under definable bijections.

\begin{theorem}\label{t3general}
Suppose $D\in \mathrm{Def}(M^n)$ and $f:D\rightarrow M^n$ is a definable injection. Then $\Lambda_{\widetilde{\mathfrak A}}(D)=\Lambda_{\widetilde{\mathfrak A}}(f(D))$ for each colour class $\widetilde{\mathfrak A}$.
\end{theorem}

Let $ev:\overline{\mathrm{Def}}(M)\rightarrow\mathbb Z[\overline{\mathcal X}^*]/\mathcal J$ be the map defined by $D\mapsto \sum\{\Lambda_{\widetilde{\mathfrak A}}(D):\widetilde{\mathfrak A}\in\overline{\mathcal X}^*/\approx\}$. This map is well defined since the sum is finite for every $D$ for reasons similar to those for the special case. Furthermore $ev_{D_1}=ev_{D_2}$ whenever $D_1$ and $D_2$ are definably isomorphic since $\Lambda_{\widetilde{\mathfrak A}}(D_1)=\Lambda_{\widetilde{\mathfrak A}}(D_2)$  for each colour class $\widetilde{\mathfrak A}$. In fact $ev$ is a semiring homomorphism. The proof of the following theorem is analogous to that of theorem \ref{t5}.

\begin{theorem}\label{t5general}
The map $ev:\widetilde{\mathrm{Def}}(M)\rightarrow\mathbb Z[\overline{\mathcal X}^*]/\mathcal J$ defined by $[D]\mapsto ev_{[D]}$ is a semiring homomorphism.
\end{theorem}

The final step in the proof of \ref{FINALgeneral} is the following analogue of lemma \ref{INJEV}.
\begin{lemma}\label{INJEVgeneral}
The map $\widetilde{ev}:\widetilde{\widetilde{\mathrm{Def}}(M)}\rightarrow\mathbb Z[\overline{\mathcal X}^*]/\mathcal J$ is injective.
\end{lemma}

\begin{proof} The proof of this lemma needs some modification of the first paragraph of the proof of lemma \ref{INJEV} in order to incorporate the invariants ideal. Let $\mathcal U:= \{\alpha\in\overline{\mathcal A^d}: A_1\cap A_2=\emptyset$ for all distinct $A_1,A_2\in\alpha\}$.

If $ev_{[\alpha]}=ev_{[\beta]}$ for some $\alpha,\beta\in\mathcal U$, then we can obtain two antichains $\alpha'\in[\alpha]\cap\mathcal U,\beta'\in[\beta]\cap\mathcal U$ such that $\bigcup\alpha=\bigcup\alpha',\bigcup\beta=\bigcup\beta'$ and $\{\alpha',\beta'\}$ is compatible. Hence we have $\Lambda_{\widetilde{\mathfrak A}}(\alpha)=\Lambda_{\widetilde{\mathfrak A}}(\alpha')$, $\Lambda_{\widetilde{\mathfrak A}}(\beta)=\Lambda_{\widetilde{\mathfrak A}}(\beta')$ for each colour class $\widetilde{\mathfrak A}$. Observe that the equalities, if considered in the codomain ring, are modulo the invariants ideal. Now $\Lambda_{\widetilde{\mathfrak A}}(\alpha')=|\alpha'\cap(\bigcup\widetilde{\mathfrak A})|\delta_{[[\mathbf{P}(\alpha')]]}$, where $\mathbf{P}$ is the only band (if exists) such that $\mathbf{P}\cap\alpha'\cap(\bigcup\widetilde{\mathfrak A})\neq\emptyset$. Since $\mathbf{P}(\alpha')=\mathbf{P}(\beta')$ for each such colour class by the definition of compatibility, we get $|\alpha\cap(\bigcup\widetilde{\mathfrak A})|=|\beta\cap(\bigcup\widetilde{\mathfrak A})|$ for each colour class $\widetilde{\mathfrak A}$.

A definable isomorphism can be easily constructed between the $pp$-convex sets represented by $\alpha'$ and $\beta'$, which are the sets represented by $\alpha$ and $\beta$ respectively. The rest of the proof is similar to the proof of \ref{INJEV}.
\end{proof}

\begin{proof} (Theorem \ref{FINALgeneral}) We have shown that the map $\widetilde{ev}$ is injective in the previous lemma. Then we observe that the sets of the form $\bigcup\alpha$ for some $\alpha\in\mathcal U$ are capable of producing every element of the quotient ring $\mathbb Z[\overline{\mathcal X}^*]/\mathcal J$ of the form $\sum n_{\mathfrak A}\delta_{\mathfrak A}+\mathcal J$, where the nonzero coefficients are positive. This completes the proof by an argument similar to the proof of theorem \ref{FINAL}.
\end{proof}

Since the Grothendieck ring is a quotient ring, we do not necessarily know if it is nontrivial. But the following corollary of theorem \ref{FINALgeneral} shows this result, proving Prest's conjecture in full generality.
\begin{cor}\label{MAINRESULTgeneral}
If $M$ is a nonzero right $\mathcal R$-module, then there is a split embedding $\mathbb Z\rightarrowtail K_0(M)$.
\end{cor}
\begin{proof}
Consider the colour class $\widetilde{\mathfrak U}$, where $\mathfrak U$ is the identity element of the monoid $\overline{\mathcal X}^*$. A $pp$-set $P$ is an element of $\bigcup\widetilde{\mathfrak U}$ if and only if $P$ is finite. Finite sets enjoy the special property that two finite sets are isomorphic to each other if and only their cardinalities are equal. Furthermore, every such isomorphism is definable. In particular, $\mathcal R(\widetilde{\mathfrak U})\cong\mathbb Z$ if $M$ is a nonzero module. Next we observe that the set $\bigcup\widetilde{\mathfrak U}$ is closed under multiplication and hence the colour class group $\mathcal R(\widetilde{\mathfrak U})$ can be given the structure of a quotient of the monoid ring $\mathbb Z[\bigcup\widetilde{\mathfrak U}]$ with certain relations, where the multiplicative relations of the monoid ring are finitary and hence already present in the relations for $\mathcal R(\widetilde{\mathfrak U})$. We have thus described the ring structure of $\mathcal R(\widetilde{\mathfrak U})$ and this ring is naturally a subring of $K_0(M)$.

To complete the proof, we show that the map $\pi_0:K_0(M)\rightarrow\mathcal R(\widetilde{\mathfrak U})$ given by $\sum_{\widetilde{\mathfrak A}\in(\overline{\mathcal X}^*/\approx)} n_{\widetilde{\mathfrak A}}\delta_{\widetilde{\mathfrak A}}\mapsto n_{\widetilde{\mathfrak U}}\delta_{\widetilde{\mathfrak U}}$ is a surjective ring homomorphism.

The map $\pi_0$ is clearly an additive group homomorphism. Note that the multiplicative monoid $\bigcup\widetilde{\mathfrak U}$ is a sub-monoid of $\overline{\mathcal X}^*$. Also note that $\mathcal J(\widetilde{\mathfrak A})\cap\mathcal J(\widetilde{\mathfrak B})=\emptyset$ if $\widetilde{\mathfrak A}\neq\widetilde{\mathfrak B}$. Furthermore, $\mathfrak A\star\mathfrak B\in\bigcup\widetilde{\mathfrak U}$ if and only if $\mathfrak A,\mathfrak B\in\bigcup\widetilde{\mathfrak U}$. Thus the coefficient of $\delta_{\widetilde{\mathfrak U}}$ in the product of two elements of $K_0(M)$ is determined by the coefficient of $\delta_{\widetilde{\mathfrak U}}$ of the individual elements. Hence $\pi_0$ is also multiplicative. The surjectivity is clear. This completes the proof.
\end{proof}

Now we can give a proof that the Grothendieck ring of a module is an invariant of its theory.

\begin{proof} (Proposition \ref{eleequivmod})
Elementarily equivalent modules have isomorphic lattices of $pp$-sets and they also satisfy the same invariant conditions (see \cite[Corollary\,2.18]{PreBk}). Hence theorem \ref{FINALgeneral} yields the result.
\end{proof}

\section{Applications}\label{appl}
\subsection{Pure embeddings and Grothendieck rings}\label{pure}

We will investigate some categorical properties of Grothendieck rings of modules in this section. The main aim is to prove the following theorem.

\begin{theorem}\label{puresurj}
Let $i:N\rightarrow M$ be a pure embedding of right $\mathcal R$-modules such that the theory of $M$ satisfies $Th(M)=Th(M)^{\aleph_0}$. Then $i$ induces a surjective ring homomorphism $I:K_0(M)\twoheadrightarrow K_0(N)$.
\end{theorem}

This theorem will be proved using a series of results of functorial nature. We begin with the definition of a pure embedding.

\begin{definition}
Let $M$ be a right $\mathcal R$-module. A submodule $N\leq M$ is called a \textbf{pure submodule} if, for each $n$, $A\cap N^n\in\mathcal L_n^\circ(N)$ for every $A\in\mathcal L_n^\circ(M)$.\\
A monomorphism $i:N\rightarrow M$ is said to be a \textbf{pure monomorphism} if $iN$ is a pure submodule of $M$.
\end{definition}

The following lemma states that a pure embedding induces a map of lattices of $pp$-formulas.
\begin{lemma}\label{purelat}(see \cite[Lemma\,3.2.2]{PrePSL})
If $i:N\rightarrow M$ is a pure embedding then, for each $n$, the natural map $\overline{i}:\mathcal L_n^\circ(M)\rightarrow\mathcal L_n^\circ(N)$ given by $\overline{i}(A)=A\cap N^n$ is a surjection of lattices.
\end{lemma}

Now we state the following result about integral monoid rings.
\begin{pro}\label{monringfunc}(see \cite[II,\,Proposition\,3.1]{Lang})
Let $\Phi:A\rightarrow B$ be a homomorphism of monoids. Then there exists a unique homomorphism $h:\mathbb Z[A]\rightarrow\mathbb Z[B]$ such that $h(x)=\Phi(x)$ for all $x\in A$ and $h(1)=1$. Furthermore, $h$ is surjective if $\Phi$ is so.
\end{pro}

\begin{cor}
A pure embedding $i:N\rightarrow M$ induces a surjective homomorphism $\mathfrak i:\mathbb Z[\overline{\mathcal X}^*(M)]\twoheadrightarrow\mathbb Z[\overline{\mathcal X}^*(N)]$ of rings.
\end{cor}

\begin{proof}
Observe that every colour $\mathfrak A\in\overline{\mathcal X}^*$ has a representative in $\overline{\mathcal L}^\circ:=\bigcup_{n=1}^\infty\mathcal L_n^\circ$. Thus we get an induced surjective homomorphism $\overline{\mathcal X}^*(M)\twoheadrightarrow\overline{\mathcal X}^*(N)$ of the colour monoids using lemma \ref{purelat}. Then proposition \ref{monringfunc} yields the required surjective map of the integral monoid rings.
\end{proof}

\begin{proof} (Theorem \ref{puresurj})
Observe that since $Th(M)=Th(M)^{\aleph_0}$ holds, theorem \ref{FINAL} gives $K_0(M)\cong\mathbb Z[\overline{\mathcal X}^*(M)]$. By theorem \ref{FINALgeneral}, we have $K_0(N)\cong\mathbb Z[\overline{\mathcal X}^*(N)]/\mathcal J(N)$. Let $\pi:\mathbb Z[\overline{\mathcal X}^*(N)]\twoheadrightarrow K_0(N)$ denote the natural quotient map. Take $I=\pi\circ\mathfrak i$, where $\mathfrak i$ is the map from the previous corollary, to finish the proof.
\end{proof}

We will see an example at the end of the next section to see that theorem \ref{puresurj} fails if $Th(M)\neq Th(M)^{\aleph_0}$.

Recall that the notation $M^{(\aleph_0)}$ denotes the direct sum of countably many copies of a module $M$. It follows immediately from \cite[Lemma\,2.23(c)]{PreBk}) that the lattices $\mathcal L_1(M)$ and $\mathcal L_1(M^{(\aleph_0)})$ are isomorphic and $T:=Th(M^{(\aleph_0)})$ satisfies $T=T^{\aleph_0}$. We summarize these observations in the following corollary of theorem \ref{puresurj}.

\begin{cor}
Let $i_n:M\rightarrow M^{(\aleph_0)}$ denote the natural embedding of $M$ onto the $n^{th}$ component of $M^{(\aleph_0)}$. Then $i_n$ induces the natural quotient map $K_0(M^{(\aleph_0)})=\mathbb Z[\overline{\mathcal X}^*(M)]\twoheadrightarrow\mathbb Z[\overline{\mathcal X}^*(M)]/\mathcal J(M)=K_0(M)$.
\end{cor}

For a ring $\mathcal R$, let $\rm Mod\mbox{-}\mathcal R$ denote the category of right $\mathcal R$-modules. The theory $Th({\rm Mod\mbox{-}\mathcal R})$ is not a complete theory. But we may take a canonical complete theory extending it as follows. Recall that Grothendieck rings of elementarily equivalent modules are isomorphic by proposition \ref{eleequivmod}. Equivalently, $K_0(M)$ is determined by $Th(M)$ which, in turn, is determined by its invariants conditions (theorem \ref{FINALgeneral}).
\begin{definition}
Let $P$ be a direct sum of one model of each complete theory of right $\mathcal R$-modules. Then $T^*=Th(P)$ is referred to as \textbf{the largest complete theory of right $\mathcal R$-modules}.
\end{definition}
Thus every right $\mathcal R$-module is elementarily equivalent to a direct summand of some model of $Th(P)$. Now we note the following result without proof and define the Grothendieck ring of the module category.
\begin{deflem}(see \cite[6.1.1,\,6.1.2]{Perera})\label{GrRngModCat}
Let $T^*$ denote the largest complete theory of right $\mathcal R$-modules. Then $T^*=(T^*)^{\aleph_0}$. Furthermore if $P_1$ and $P_2$ are both direct sums of one model of each complete theory of right $\mathcal R$-modules, then $K_0(P_1)\cong K_0(P_2)$. We define the \textbf{Grothendieck ring of the module category}, denoted $K_0({\rm Mod\mbox{-}\mathcal R})$, to be the Grothendieck ring of the largest complete theory of right $\mathcal R$-modules.
\end{deflem}

As a consequence of theorem \ref{puresurj}, we state a result connecting Grothendieck rings of individual modules with that of the module category.
\begin{cor}
Let $M$ be a right $\mathcal R$-module. Then $K_0(M)$ is a quotient of $K_0({\rm Mod\mbox{-}\mathcal R})$.
\end{cor}
\begin{proof} Let $T^*$ be the largest complete theory of right $\mathcal R$-modules. Then lemma \ref{GrRngModCat} gives that, for any $P\models T^*$, $Th(P)=T^*$ satisfies $T^*=(T^*)^{\aleph_0}$ and we also have $K_0(P)\cong K_0({\rm Mod\mbox{-}\mathcal R})$.

By the definition of $T^*$, there is a module $M'$ elementarily equivalent to $M$ such that $M'$ is a direct summand of $P$. Since the embedding $M'\rightarrowtail P$ is pure, we get a surjective homomorphism $K_0(P)\twoheadrightarrow K_0(M')$. Thus the required quotient map is the composite $K_0(\mathrm{Mod\mbox{-}\mathcal R})\cong K_0(P)\twoheadrightarrow K_0(M')\cong K_0(M)$, where the last isomorphism is obtained from proposition \ref{eleequivmod}.
\end{proof}

\subsection{Torsion in Grothendieck rings}\label{tors}

As an application of the structure theorem for Grothendieck rings, theorem \ref{FINALgeneral}, we provide an example of a module whose Grothendieck ring contains a nonzero torsion element (i.e. a nonzero element $a$ such that $na=0$ for some $n\geq 1$). We also calculate the Grothendieck ring $K_0(\mathbb Z_\mathbb Z)$.

\begin{definition}
The \textbf{ring of $p$-adic integers}, denoted $\mathbb Z_p$, is the inverse limit of the system $\hdots\twoheadrightarrow\mathbb Z/p^n\mathbb Z\twoheadrightarrow\hdots\twoheadrightarrow\mathbb Z/p^2\mathbb Z\twoheadrightarrow\mathbb Z/p\mathbb Z\twoheadrightarrow 0$.\\
\end{definition}

The ring $\mathbb Z_p$ is a commutative local PID with the ideal structure given by
\begin{equation*}
\mathbb Z_p\supsetneq p\mathbb Z_p\supsetneq\hdots\supsetneq p^n\mathbb Z_p\supsetneq\hdots\supsetneq 0.
\end{equation*}

In particular, $\mathbb Z_p$ is a commutative noetherian ring and hence satisfies the hypothesis of the following proposition.
\begin{pro}(see \cite[p.19,\,Ex.\,2(ii)]{PreBk})
If $\mathcal R$ is a commutative noetherian ring then the $pp$-definable subgroups of the module $\mathcal R_\mathcal R$ are precisely the finitely generated ideals of $\mathcal R$.
\end{pro}

It can be observed that the maps $t_n:\mathbb Z_p\rightarrow p^n\mathbb Z_p$ which are `multiplication by $p^n$' are $pp$-definable isomorphisms for each $n\geq 1$. Thus a simple computation shows that the monoid of colours, $\overline{\mathcal X}^*(\mathbb Z_p)$, is isomorphic to the monoid $\mathbb N$.

If $X$ denotes the class of $\mathbb Z_p$ in $K_0(\mathbb Z_p)$, then the invariants ideal $\mathcal J(\mathbb Z_p)$ is generated by the relations $\{X=p^nX:n\geq 1\}$. The relation $(p^n-1)X=0$ is an integral multiple of the relation $(p-1)X=0$ for each $n\geq 1$. Thus $\mathcal J(\mathbb Z_p)$ is principal and generated by the single relation $(p-1)X=0$. We summarize this discussion as the following corollary to theorem \ref{FINALgeneral}.
\begin{cor}\label{p-adic}
Let $\mathbb Z_p$ denote the ring $p$-adic integers. Then \center{$K_0(\mathbb Z_p)\cong\mathbb Z[X]/\langle(p-1)X\rangle$.}
\end{cor}

Consider the split (hence pure) embedding $i:\mathbb Z_p^{(2)}\rightarrowtail\mathbb Z_p^{(3)}$ of $\mathbb Z_p$-modules given by $(a,b)\mapsto(a,b,0)$, where $M^{(k)}$ denotes the direct sum of $k$ copies of $M$. We want to show that this embedding witnesses the failure of theorem \ref{puresurj} since the theory $T:=Th(\mathbb Z_p^{(3)})$ of the target module doesn't satisfy the condition $T=T^{\aleph_0}$. The following proposition is helpful for the calculation of Grothendieck rings.

\begin{pro}(see \cite[Lemma\,2.23]{PreBk})
If $\phi(x)$ and $\psi(x)$ denote $pp$-formulas, then
\begin{enumerate}
  \item $\phi(M\oplus N)=\phi(M)\oplus\phi(N)$,
  \item $\mathrm{Inv}(M\oplus N;\phi,\psi)=\mathrm{Inv}(M;\phi,\psi)\mathrm{Inv}(N;\phi,\psi)$.
\end{enumerate}
\end{pro}

It is clear that the induced map $\mathfrak{i}:\mathbb Z[\overline{\mathcal X}^*(\mathbb Z_p^{(3)})]\rightarrow\mathbb Z[\overline{\mathcal X}^*(\mathbb Z_p^{(2)})]$ is the identity map on $\mathbb Z[X]$ since $\mathbb Z[\overline{\mathcal X}^*(\mathbb Z_p^{(k)})]\cong K_0(\mathbb Z_p^{(\aleph_0)})\cong\mathbb Z[X]$ for any $k\geq 1$. Further the previous proposition shows that $\mathcal J(\mathbb Z_p^{(k)})=\langle(p^k-1)X\rangle$ for any $k\geq 1$. Since $\mathcal J(\mathbb Z_p^{(3)})\nsubseteq\mathcal J(\mathbb Z_p^{(2)})$, there is no surjective map $K_0(\mathbb Z_p^{(3)})\twoheadrightarrow K_0(\mathbb Z_p^{(2)})$.

\textbf{The abelian group of integers}: Since the ring $\mathbb Z$ is a commutative PID, the $pp$-definable subgroups of the module $\mathbb Z_\mathbb Z$ are precisely the ideals $n\mathbb Z$ for $n\geq 0$. Thus the monoid $\overline{\mathcal X}^*(\mathbb Z)$ is isomorphic to $\mathbb N$. Furthermore if $X$ denotes the class of $\mathbb Z$ in $K_0(\mathbb Z)$, the invariants ideal is generated by the relations $X=nX$ for each $n\geq 1$. This forces $\mathcal J(\mathbb Z)=\langle X\rangle$ and thus $K_0(\mathbb Z_\mathbb Z)\cong\mathbb Z$.

\subsection{Representing definable sets uniquely}\label{CDT}

We fix some $\mathcal R$-module $M$ whose theory $T$ satisfies the condition $T=T^{\aleph_0}$ and some $n\geq 1$. As usual we drop all the subscripts $n$ and write $\mathcal L\setminus\{\emptyset\},\mathcal A\setminus\{\emptyset\},\hdots$ as $\mathcal L^*,\mathcal A^*,\hdots$ respectively.

The $pp$-elimination theorem for the model theory of modules (theorem \ref{PPET}) states that every definable set can be written as a finite disjoint union of blocks. But this representation is far from being unique in any sense. On the other hand we have unique representations for $pp$-convex sets (proposition \ref{UNIQUE1}) and cells (lemma \ref{UNIQUE2}). We exploit these ideas to achieve a unique representation for every definable set - an expression as a disjoint union of cells. This result will be called the `cell decomposition theorem'.

We begin by defining some terms useful to describe the cell decomposition theorem.
\begin{definition}
Let $\mathcal F=\{C_j\}_{j=1}^l\subseteq\mathcal C$ be a family of pairwise disjoint cells. If there is a permutation $\sigma$ of $[l]$ such that $P(C_{\sigma(j+1)})\prec N(C_{\sigma(j)})$ for $1\leq j\leq l-1$, then we say that the family $\mathcal F$ is a \textbf{tower of cells}. We call the number $l$ the \textbf{height} of the tower. We denote the set of all finite towers of cells by $\mathcal T$. We define a function $\zeta:\mathcal T\rightarrow \mathbb N$ which assigns its height to a tower.
\end{definition}

\begin{definition}
Let $\alpha_i\in\mathcal A^*$ for $1\leq i\leq k$. If $\alpha_{i+1}\prec\alpha_{i}$ for each $1\leq i\leq k-1$, we say that $\overline\alpha=\{\alpha_i\}_{i=1}^k$ is a $\prec$-\textbf{chain}. We denote the set of all finite $\prec$-chains in $\mathcal A^*$ by $\mathcal W$. We define a function $\omega:\mathcal W\rightarrow \mathbb N$, which assigns \textbf{height} to each $\prec$-chain, by $\omega(\overline\alpha)=\lceil(\frac{|\overline\alpha|}{2})\rceil$ where $\lceil q\rceil$ is the smallest integer larger than or equal to $q$.
\end{definition}

The following proposition states that towers and chains are two different ways of expressing the same kind of object.
\begin{pro}
There is a bijection $\Phi:\mathcal T\rightarrow\mathcal W$ preserving height i.e., $\omega(\Phi(\mathcal F))=\zeta(\mathcal F)$ for every $\mathcal F\in\mathcal T$.
\end{pro}
\begin{proof} Let $\mathcal F=\{C_j\}_{j=1}^l$ be a tower of cells with height $l$. Without loss, we may assume that $P(C_j)\prec N(C_{j+1})$, i.e., the associated permutation is the identity. We first define a non-negative integer $k$ as follows.

\begin{math}
k=\begin{cases}
    0, & \mbox{if } l=0, \\
    2l+1, &\mbox{if } l>0\mbox{ and } N(C_l)=\emptyset,\\
    2l+2, &\mbox{if } l>0\mbox{ and } N(C_l)\neq\emptyset.
\end{cases}
\end{math}

For each $1\leq i\leq k$, we define an antichain $\alpha_i\in\mathcal A^*$ as follows.

\begin{math}
\alpha_i=\begin{cases}
    P(C_j), & \mbox{if } i=2j+1,\\
    N(C_j), &\mbox{if } i=2j+2.
\end{cases}
\end{math}

Then $\overline\alpha=\{\alpha_i\}_{i=1}^k$ is clearly a $\prec$-chain and the map $\Phi(\mathcal F):=\overline{\alpha}$ can be easily checked to be injective.

To prove surjectivity, let $\overline\beta\in\mathcal W$. We modify $\overline\beta$ to obtain a $\prec$-chain $\overline\beta'=\{\beta_i'\}_{i=1}^{2\omega(\overline\beta)}$ in $\mathcal A$ as follows.

\begin{math}
\beta_i'=\begin{cases}
    \beta_i, & \mbox{if } 1\leq i\leq |\overline\beta|,\\
    \emptyset, &\mbox{if } |\overline\beta|\neq 2\omega(\overline\beta)\mbox{ and } i=2\omega(\overline\beta).
\end{cases}
\end{math}

Then $|\overline\beta'|$ is an even integer. We define $C'_j=\bigcup\beta_{2j+1}\setminus\bigcup\beta_{2j+2}$ for $1\leq j\leq |\overline\beta'|/2$. The family $\mathcal F':=\{C'_j\}_{j=1}^{|\overline\beta'|/2}$ clearly satisfies $\Phi(\mathcal F')=\overline{\beta}$.

The height preservation property is easy to check from the explicit constructions above.
\end{proof}

\begin{pro}
Let $\{A_i\}_{i=1}^m\in\mathcal P$ and $B\in\mathcal B$ be such that $B\subseteq\bigcup_{i=1}^m A_i$. Then $\overline B\subseteq\bigcup_{i=1}^m A_i$.
\end{pro}
\begin{proof}
We have $\overline B=B\cup\bigcup N(B)$. Hence $\overline B\subseteq\bigcup_{i=1}^m A_i\cup\bigcup N(B)$. By \ref{NLU}, $\overline B\subseteq A_i$ for some $i$, or $\overline B\subseteq D$ for some $D\in N(B)$. The latter case is not possible since $N(B)\prec P(B)=\{\overline B\}$. Hence the result.
\end{proof}

\begin{lemma}\label{CHAINANTITOWER}
Let $D\in \mathrm{Def}(M^n)$. Then there is a unique $pp$-convex set $\overline D$ which satisfies $D\subseteq\bigcup\alpha\ \Rightarrow\ \overline D\subseteq\bigcup\alpha$ for every $\alpha\in\mathcal A$.
\end{lemma}
\begin{proof} Let $D=\bigsqcup_{i=1}^m B_i=\bigsqcup_{j=1}^l B'_j$ be any two representations of $D$ as disjoint unions of blocks.

\textbf{Claim}: $\bigcup_{i=1}^m \overline{B_i}=\bigcup_{j=1}^l \overline{B'_j}$

Proof of the claim: We have $B_i\subseteq\bigsqcup_{i=1}^m B_i=\bigsqcup_{j=1}^l B'_j\subseteq\bigcup_{j=1}^l \overline{B'_j}$ for each $i$. Hence $\overline{B_i}\subseteq\bigcup_{j=1}^l \overline{B'_j}$ by the previous proposition. Therefore $\bigcup_{i=1}^m \overline{B_i}\subseteq\bigcup_{j=1}^l \overline{B'_j}$. The reverse containment is by symmetry and hence the claim.

Now we define $\overline D=\bigcup_{i=1}^m \overline{B_i}$. By the claim, this $pp$-convex set is uniquely defined.

Let $\alpha\in\mathcal A$ be such that $D\subseteq\bigcup\alpha$. But $D=\bigsqcup_{i=1}^m B_i$. Hence $B_i\subseteq\bigcup\alpha$ for each $i$. By arguments similar to the proof of the claim, we get $\bigcup_{i=1}^m \overline{B_i}\subseteq\bigcup\alpha$ i.e., $\overline D\subseteq\bigcup\alpha$.
\end{proof}

The assignment $D\mapsto\overline{D}$, where $\overline{D}$ is the $pp$-convex set obtained from the lemma, defines a closure operator $\mathrm{Def}(M^n)\rightarrow\mathcal A_n$. This closure operation is extremely useful in proving the cell decomposition theorem.

\begin{theorem}\label{CDT1}
\textbf{Cell Decomposition Theorem}: There is a bijection between the set $\mathrm{Def}(M^n)$ of all definable subsets of $M^n$ and the set $\mathcal T$ of towers of cells.
\end{theorem}
\begin{proof}
Let $D\in \mathrm{Def}(M^n)$. We construct a tower $\mathcal F$ of cells by defining a nested sequence $\{D_j\}_{j\geq 0}$ of definable subsets of $D$ as follows.

We set $D_0:=D$ and, for each $j>0$, we set $D_j:=D_{j-1}\setminus C_j$, where  $C_j:=\overline{D_{j-1}}\setminus (\overline{\overline{D_{j-1}}\setminus D_{j-1}})$ is a cell. We stop this process when we obtain $D_j=\emptyset$ for the first time. This process must terminate because the elements of the antichains involved in this process belong to some finite nest containing a fixed decomposition of $D$ into blocks.

In the converse direction, we assign $\bigcup\mathcal F\in \mathrm{Def}(M^n)$ to $\mathcal F\in\mathcal T$.

It is easy to verify that the two assignments defined above are actually inverses of each other.
\end{proof}

The following corollary combines theorem \ref{CDT1} with lemma \ref{CHAINANTITOWER} and gives a combinatorial representation theorem for $\mathrm{Def}(M^n)$, which roughly states that every definable subset of $M^n$ can be represented uniquely as a finite $\prec$-chain in the free distributive lattice $\mathcal A$ over the meet semilattice $\mathcal L$.
\begin{cor}
There is a bijection between the set $\mathcal W_n$ of finite chains in $\mathcal A_n^*$ and $\mathrm{Def}(M^n)$.
\end{cor}

\subsection{Connectedness}\label{C}

We fix a right $\mathcal R$-module $M$ satisfying $Th(M)=Th(M)^{\aleph_0}$ and some $n\geq 1$. We drop all the subscripts $n$ as usual.

Recall that every global characteristic of a definable set is preserved under definable isomorphisms (Theorem \ref{t4}). In this section we describe what we mean by the statement that a definable subset of a (finite power of a) module is connected. The property of being connected is not preserved under definable isomorphisms. We prove a (topological) property of connected sets which states that a definable connected set $A$ contained in another definable set $B$ is in fact contained in a connected component of $B$.

Let $\mathcal{F},\mathcal{F}'\subseteq\mathcal{B}$ be two finite families of disjoint blocks such that $\bigcup\mathcal{F}=\bigcup\mathcal{F}'$. Then we say that $\mathcal{F}'$ is a \textbf{refinement} of $\mathcal{F}$ if for each $F'\in\mathcal{F}'$, there is a unique $F\in\mathcal{F}$ such that $F'\subseteq F$. Recall from \ref{CH1} that if $\bigcup\mathcal{F}\in\mathcal{B}$ and if $\mathcal{D}$ is the corresponding nest, then $\{\mathrm{Core}_\mathcal{D}(D)\}_{D\in\mathcal{D}^+}$ is a refinement of $\mathcal{F}$, where $\mathcal{D}^+$ is the set $\delta_{\mathcal{D}}^{-1}\{1\}$. We use this property of nests to attach a digraph with each of them.

\begin{definition}
Let $\mathcal D$ be a nest corresponding to a fixed finite family of pairwise disjoint blocks. We define a \textbf{digraph structure} $\mathcal{H}(\mathcal{D}^+)$ on the set $\mathcal{D}^+$. The pair $(F_1, F_2)$ of elements of $\mathcal{D}^+$ will be said to constitute an arrow in the digraph if $F_1\subsetneq F_2$ and $F_1\subseteq F\subseteq F_2$ for some $F\in\mathcal{D}^+$ if and only if $F=F_1$ or $F=F_2$.
\end{definition}

If $\bigcup_{F\in\mathcal{D}^+}\mathrm{Core}_\mathcal{D}(F)\in\mathcal{B}$, then $\mathcal{D}^+$ is an upper set and in particular $\mathcal{H}(\mathcal{D}^+)$ is \textbf{weakly connected} i.e., its underlying undirected graph is connected. It seems natural to use this property to define the connectedness of a definable set.

\begin{definition}\label{conn}
Let $D\in\mathrm{Def}(M^n)$ be represented as $D=\bigcup\mathcal F$, where $\mathcal F\subseteq\mathcal B$ be a finite family of pairwise disjoint blocks and let $\mathcal{D}$ denote the nest corresponding to $\mathcal F$. We say that $D$ is \textbf{connected} if and only if the digraph $\mathcal{H}(\mathcal{D}'^+)$ is weakly connected for some nest $\mathcal{D}'$ containing $\mathcal{D}$.
\end{definition}

Note the existential clause in this definition. Let $\mathcal F,\mathcal F'$ be two finite families of pairwise disjoint blocks with $\bigcup\mathcal F=\bigcup\mathcal F'$ and let $\mathcal D,\mathcal D'$ denote the nests corresponding to them. If $\mathcal F'$ refines $\mathcal F$, then the number of weakly connected components of $\mathcal H(\mathcal D'^+)$ is bounded between $0$ and the number of weakly connected components of $\mathcal H(\mathcal D^+)$. This observation allows us to define the following invariant.

\begin{definition}
We define the \textbf{number of connected components} of $D$, denoted $\lambda(D)$, for every nonempty definable set $D$ to be the least number of weakly connected components of $\mathcal H(\mathcal D^+)$, where $\mathcal D$ varies over nests refining a fixed partition of $D$ into disjoint blocks. We set $\lambda(\emptyset)=0$.
\end{definition}

In the discussion on connectedness, we have treated blocks as if they are the basic connected sets. Note that a definable set $D$ is connected if and only if $\lambda(D)=1$. We denote the set of all connected definable subsets of $M^n$ by $\mathbf{Con}_n$. We tend to drop the suffix $n$ if it is clear from the context. We have $\mathbf B_n\subseteq\mathbf{Con}_n$ as expected.

\begin{illust}
Consider the vector space $\mathbb R_\mathbb R$. The $pp$-definable subsets of the plane, $\mathbb R^2$, are points and lines and the plane.

Note that if a definable subset of $\mathbb R^2$ is topologically connected, then it is connected according to definition \ref{conn}. But the converse is not true. The set $B=\{(x,0):x\neq 0\}$ is not topologically connected, but $B\in\mathbf{Con}$ since $B$ is a block.

If $D$ denotes the union of two coordinate axes with the origin removed, then the number of topologically connected components of $D$ is $4$, whereas $\lambda(D)=2$.
\end{illust}

\begin{rmk}
If $X$ is a `nice' topological space (e.g., a manifold), then the rank $\beta_0$ of the homology group $H_0(X)$ is the number of (path) connected components of $X$. To note the analogy, consider $P\in \mathcal L_n$ and $\alpha\in\mathcal L_P$. If $\alpha\neq\emptyset$, then $\beta_0(\mathcal K^P(\alpha))=\lambda(\bigcup\alpha\setminus P)$. Note that the `deleted neighbourhood' of $P$ in $\alpha$, i.e., the set $\bigcup\alpha\setminus P$, occurs in this correspondence since the `non-deleted neighbourhood' $\bigcup\alpha$ is connected.
\end{rmk}

Topologically connected sets satisfy the following property. If a connected set $A$ is contained in another set $B$, then $A$ is actually contained in a connected component of $B$. We have a similar result here.

\begin{theorem}\label{topconn}
Let $A,B_i\in\mathbf{Con}$ for $1\leq i\leq m$ be such that $\lambda(\bigcup_{i=1}^mB_i)=m$. If $A\subseteq\bigcup_{i=1}^mB_i$, then $A\subseteq B_i$ for a unique $i$.
\end{theorem}
\begin{proof}
Let $\mathcal D$ be a nest containing the nests corresponding to some fixed families of blocks partitioning $A$ and all the $B_i$. The restriction of the digraph $\mathcal H(\mathcal D^+)$ to $A$ is a subdigraph of $\mathcal H(\mathcal D^+)$. Since the former is weakly connected, it is a sub-digraph of exactly one of the $m$ weakly connected components of the latter.
\end{proof}

\subsection{Remarks and questions}\label{rmkcom}

Consider the structure of the proof of the special case of the main theorem. Manipulation of different lattice-like structures is one of the important themes in this paper. The partial quantifier elimination result for theories of modules (theorem \ref{PPET}) makes the meet-semi-lattice $\mathcal L_n$, of $pp$-definable sets, the basic object of study. The lattice of antichains $\mathcal A_n$ is the free distributive lattice on $\mathcal L_n$ and simplicial methods are natural for studying the `set-theoretic geometry' associated with antichains. The local processes in $\mathrm{Def}(M^n)$ are similar to, but independent from, the local processes in $\mathrm{Def}(M^m)$ when $n\neq m$ and these different `dimensions' start to interact with each other only when we are concerned with the multiplicative structure. The fact that the $pp$-sets are closed under projections is not directly relevant to the technique.

Note that the model-theoretic condition $\rm T=T^{\aleph_0}$ is equivalent to the lattice-theoretic statement that every element of $\mathcal L_n$ considered as an element of the lattice $\mathcal A_n$ is `join-irreducible'. The unique representation theorem (theorem \ref{CDT1}) relies solely on this fact and in particular this is a statement about lattices of sets. We would like to know if this idea can be expressed in some more abstract setting.

The algebraic K-theory functor ${\rm K_0:Ring\rightarrow Ab}$ is covariant, whereas the model-theoretic Grothendieck ring functor $K_0$ is contravariant on pure embeddings (theorem \ref{puresurj}). Note that $K_0(M)$ depends on $\overline{\mathcal L}(M)$ in a covariant way and the assignment $M\mapsto\overline{\mathcal L}(M)$ is contravariant. This strongly suggests that the answer to the following question is positive.

\begin{que}
Is there a way to define the Grothendieck ring for a sequence $(L_n)_{n\geq 0}$ of meet-semi-lattices (with inclusion and projection maps) under certain conditions in a way that is abstractly similar to the technique used in the proof of theorem \ref{FINAL}?
\end{que}

A more specific question could be asked for model-theoretic Grothendieck rings.
\begin{que}
Are there any structures admitting some form of quantifier elimination, whose Grothendieck rings can be computed using a similar technique?
\end{que}

Though there are modules with additive torsion elements in Grothendieck rings (corollary \ref{p-adic}), we believe that there are no examples with non-trivial multiplicative torsion elements (i.e. elements $a\in K_0(M)$ such that $a^n=1$ for some $n>1$).
\begin{conj}
There are precisely two units (namely $\pm 1$) in the Grothendieck ring $K_0(M)$ of a nonzero module $M$.
\end{conj}

\subsection*{Acknowledgement}

I would like to thank Prof.~Mike Prest for introducing me to this topic, for valuable long discussions and for careful reading of the paper.

\end{document}